\documentclass[10pt]{amsart}

\usepackage[margin=1.5in]{geometry} 
\usepackage{amsmath, amsfonts, amsthm, amstext, amssymb} 
\usepackage{hyperref} 
\usepackage{todonotes} 
\usepackage[nameinlink, capitalize]{cleveref} 
\usepackage{tikz-cd} 
\usepackage[shortlabels]{enumitem} 

\numberwithin{equation}{section} 
\numberwithin{table}{section} 

\theoremstyle{definition}
\newtheorem{definition}[equation]{Definition}
\newtheorem{remark}[equation]{Remark}
\newtheorem{example}[equation]{Example}
\newtheorem{notation}[equation]{Notation}
\theoremstyle{theorem}
\newtheorem{lemma}[equation]{Lemma}
\newtheorem{proposition}[equation]{Proposition}
\newtheorem{theorem}[equation]{Theorem}
\newtheorem{corollary}[equation]{Corollary}

\newtheorem{thmx}{Theorem} 
 
\newtheorem{corx}[thmx]{Corollary}

\newcommand{\mhyphen}{\textnormal{-}}

\newcommand{\ul}[1]{\underline{#1}}

\def\g{\mathbb{G}}
\def\n{\mathbb{N}}
\def\q{\mathbb{Q}}
\def\z{\mathbb{Z}}

\def\ca{\mathcal{A}}
\def\cc{\mathcal{C}}
\def\cd{\mathcal{D}}
\def\cf{\mathcal{F}}
\def\co{\mathcal{O}}
\def\cp{\mathcal{P}}

\def\uTor{\underline{\operatorname{Tor}}}
\def\uA{\underline{A}}

\def\uM{\underline{M}}

\def\uN{\underline{N}}
\def\uP{\underline{P}}
\def\uR{\underline{R}}
\def\uS{\underline{S}}
\def\uZ{\underline{\mathbb{Z}}}
\newcommand{\uI}{\underline{I}}

\DeclareMathOperator{\Mod}{\mathcal{M}od}
\DeclareMathOperator{\Ab}{\mathcal{A}b}
\DeclareMathOperator{\Set}{\mathcal{S}et}
\DeclareMathOperator{\Mack}{\mathcal{M}ack}
\DeclareMathOperator{\Tamb}{\mathcal{T}amb}
\DeclareMathOperator{\Sp}{\mathcal{S}p}
\DeclareMathOperator{\Sub}{Sub}
\DeclareMathOperator{\Orb}{\mathcal{O}{rb}}
\DeclareMathOperator{\FixN}{Fix_N}
\DeclareMathOperator{\Ind}{Ind}

\newcommand{\triv}{{\textnormal{triv}}} 
\newcommand{\cplt}{{\textnormal{cplt}}} 
\newcommand{\gen}{{\textnormal{gen}}}
\newcommand{\Iso}{{\textnormal{iso}}}

\DeclareMathOperator{\id}{id} 
\DeclareMathOperator{\im}{im} 
\DeclareMathOperator{\Sym}{Sym} 
\DeclareMathOperator{\coeq}{coeq}
\DeclareMathOperator{\res}{res}
\DeclareMathOperator{\tr}{tr}
\DeclareMathOperator{\nm}{nm}
\DeclareMathOperator{\FP}{FP}
\DeclareMathOperator{\pr}{pr}
\DeclareMathOperator{\Inf}{Inf}

\begin{document}

\author{Michael A. Hill}\thanks{The first author was supported by NSF Grant DMS-1811189.}\address{Department of Mathematics, University of California Los Angeles, Los Angeles, California, U.S.A.}\email{mikehill@math.ucla.edu}
\author{David Mehrle}\address{Department of Mathematics, Cornell University, Ithaca, NY, U.S.A.}\email{dfm223@cornell.edu}
\author{J.D. Quigley}\address{Department of Mathematics, Cornell University, Ithaca, NY, U.S.A.}\email{jdq27@cornell.edu}
\title{Free incomplete Tambara functors are almost never flat}
\maketitle

\begin{abstract}
Free algebras are always free as modules over the base ring in classical algebra. In equivariant algebra, free incomplete Tambara functors play the role of free algebras and Mackey functors play the role of modules. Surprisingly, free incomplete Tambara functors often fail to be free as Mackey functors. In this paper, we determine for all finite groups conditions under which a free incomplete Tambara functor is free as a Mackey functor. For solvable groups, we show that a free incomplete Tambara functor is flat as a Mackey functor precisely when these conditions hold. Our results imply that free incomplete Tambara functors are almost never flat as Mackey functors. However, we show that after suitable localizations, free incomplete Tambara functors are always free as Mackey functors.
\end{abstract}

\maketitle

\setcounter{tocdepth}{1}
\tableofcontents

\vspace*{-1cm}

\section{Introduction}

\subsection{Motivation and Main Theorems} 

Let $k$ be a commutative ring. The free $k$-algebra on a single generator is the polynomial algebra $k[x]$. This algebra is also a free $k$-module on the countable basis $1, x, x^2, x^3, \ldots$. We take this fact for granted in classical algebra, where it is used to great effect in homological algebra. 

In this paper, we investigate the analogous property in equivariant algebra, or algebra where commutative rings are replaced by incomplete Tambara functors \cite{BH18} and abelian groups are replaced by Mackey functors. We recall these notions in \cref{Sec:Tambara}. Equivariant algebra is an important tool in representation theory, number theory, and algebraic topology. 

In this more structured version of algebra, free algebras are \textsl{not} always free as modules. If they are not free, we might ask whether they are projective or flat instead. The goal of this paper is to understand when the underlying module of a free algebra over an incomplete Tambara functor is suitable for homological algebra - when it is free, projective, or flat. Surprisingly, we find that this is almost never the case. 

The analog of a commutative ring in equivariant algebra was defined by Blumberg and the first author in \cite{BH18}. Let $G$ be a finite group. An \emph{incomplete Tambara functor} $\uR$ assigns a commutative ring $\uR(G/H)$ to each transitive $G$-set $G/H$, together with certain structure maps (conjugations, restrictions, transfers, and norms) which are encoded by a choice of \emph{indexing category} $\co$ for $G$.   

Blumberg and the first author also introduced free algebras over incomplete Tambara functors in \cite{BH18}. Fix a group $G$, indexing category $\co$, and incomplete Tambara functor $\uR$. For each subgroup $H$ of $G$, there is an equivariant refinement of a free algebra on one generator called the \emph{free $\uR$-algebra on a generator at level $G/H$}, denoted $\uR^{\co}[x_{G/H}]$. These are the free algebras which we study in this paper. 

The most fundamental incomplete Tambara functor is the Burnside functor, denoted $\uA$, which plays the role of the integers in equivariant algebra (cf. \cref{Rmk:BurnsideInitial}). Our first main theorem provides conditions on pairs of subgroup $H$ and indexing category $\co$ under which the free algebra $\uA^{\co}[x_{G/H}]$ is free as an $\uA$-module.

\begin{thmx}[\cref{Thm:Sufficiency}]\label{Thmx:Free}
  Let $G$ be a finite group, $\co$ an indexing category for $G$, and $H$ a subgroup of $G$. If $i^*_H \co = \co^{\triv}$ (see \cref{Def:TrivialIndexingCategory}) and $G/H$ is admissible for $\co$, then the free $\co$-Tambara functor on a generator at level $G/H$, $\uA^{\co}[x_{G/H}]$, is free as an $\uA$-module.
\end{thmx}

It turns out that the conditions $i_H^*\co = \co^{\triv}$ and $G/H$ is admissible for $\co$ together imply that $H$ must be a normal subgroup of $G$. 

\begin{remark}
  As in ordinary algebra, base change preserves freeness in equivariant algebra (\cref{prop:base-change-of-free-is-free}). Since $\uA$ is the initial incomplete Tambara functor, there is a map $\uA \to \uR$ for any incomplete Tambara functor $\uR$. \cref{Thmx:Free} then implies that $\uR^{\co}[x_{G/H}]$ is free as an $\uR$-module whenever $i^*_H \co = \co^{\triv}$ and $G/H$ is admissible for $\co$. 
\end{remark}

\cref{Thmx:Free} provides sufficient conditions on the pair $(\co, H)$ under which $\uA^{\co}[x_{G/H}]$ is free. However, since we are interested in homological applications where projectivity or flatness often suffice, it is natural to ask if $\uA^{\co}[x_{G/H}]$ is projective or flat as an $\uA$-module, even if it is not free. We prove the following necessary conditions:

\begin{thmx}[\cref{Cor:TrivRes}, \cref{Prop:HNormal}, and \cref{Thm:SolvableCase}]\label{Thmx:Free2}
	Let $G$ be a finite group, $\co$ an indexing category for $G$, and $H$ a subgroup of $G$. 
	
	If $\uA^{\co}[x_{G/H}]$ is flat as an $\uA$-module, then: 
	\begin{enumerate}[(a)]
	    \item $i^*_H \co \cong \co^{\triv}$;
	    \item $H$ is a normal subgroup of $G$.
	\end{enumerate} 
	
	Moreover, if $G/H$ is solvable, then $G/H$ is admissible for $\co$. 
\end{thmx}

For solvable groups, \cref{Thmx:Free} and \cref{Thmx:Free2} give a complete description of when free incomplete Tambara functors are free, or equivalently, flat, as Mackey functors.

\begin{corx}[\cref{Thm:SolvableCase}]\label{Thmx:Solvable}
	Let \(G\) be a solvable finite group, $\co$ an indexing category for $G$, and $H$ a subgroup of $G$. The following are equivalent: 
	\begin{enumerate}[(a)]
    	\item The Mackey functor underlying the free \(\co\)-Tambara functor on a class at level \(G/H\) is flat.
	    \item The Mackey functor underlying the free \(\co\)-Tambara functor on a class at level \(G/H\) is free.
	    \item \(H\) is a normal subgroup of \(G\), \(G/H\) is admissible, and \(i_H^\ast\co=\co^{\triv}\). 
	\end{enumerate} 
\end{corx}

The conditions on $H$ and $\co$ in \cref{Thmx:Free2} are easy to check, but are quite strict. Empirical evidence for small cyclic groups (cf. \cref{App:Tables}) suggests that the proportion of free $\uA$-algebras for $C_{p^n}$ which are free as $\uA$-modules decreases as $n$ increases. In fact, we prove a stronger statement:

\begin{thmx}[\cref{Thm:Asymptotics}]\label{Thmx:Asymptotics}
	Free incomplete Tambara functors for finite groups are almost never flat as Mackey functors. 
\end{thmx}

Before discussing freeness over other base incomplete Tambara functors, we pause to discuss our motivation for studying freeness. The fact that free algebras over ordinary commutative rings are free as modules has an important application in homological algebra: derived functors, such as the cotangent complex and Shukla homology, can be computed using polynomial resolutions. 

In forthcoming work, the second and third authors investigate an equivariant version of the Hochschild--Kostant--Rosenberg (HKR) theorem \cite{HKR62, loday2013} which identifies the cotangent complex \cite{Hil17, Lee19} and derived Hochschild homology \cite{BGHL19} of certain incomplete Tambara functors. The isomorphism can be shown to hold for free algebras, but the extension to all incomplete Tambara functors would require the use of resolutions by free algebras. Unfortunately, \cref{Thmx:Asymptotics} says that resolutions by free $\uA$-algebras are almost never resolutions by free $\uA$-modules, so it is unclear if they can be used to compute derived functors.

With that said, the classical HKR theorem can be strengthened to an isomorphism between de Rham cohomology and cyclic homology if one assumes that the ground ring contains the rational numbers. Similarly, in non-equivariant algebra, a module which is not free over a commutative ring $k$ may become free after base change to an appropriate localization of $k$. 

We show that after inverting a single element in $\uA$, free incomplete Tambara functors for different indexing categories are all isomorphic as Mackey functors:

\begin{thmx}[\cref{Theorem:FreeLocalUnderlyingFree}]\label{Thmx:FreeLocalUnderlyingFree}
	Let $\underline{S}^{-1}\uA$ be the Mackey functor obtained from $\uA$ by inverting $[G/e] \in \uA(G/G)$. For any indexing category $\co$ and subgroup $H \leq G$, the free $\co$-Tambara functor $\uS^{-1}\uA^{\co}[x_{G/H}]$ is free as an $\uS^{-1}\uA$-module. 
\end{thmx}

\cref{Thmx:FreeLocalUnderlyingFree} has the important consequence that free algebras \emph{can} be used to compute nonabelian derived functors for incomplete Tambara functors when working over $\uS^{-1}\uA$. This suggests that, as in classical algebra, derived functors will be easier to understand after localization or rationalization.

Unfortunately, however, this localization is a very brutal operation. It essentially destroys all data that is not already present in the underlying level of the Mackey functor.

\begin{thmx}[\cref{prop:invertGiso}, \cref{cor:invertGAllModulesAreFixedPoint}, and \cref{theorem:invertOrderSplitsModules}]\label{Thmx:LocalBrutal}
	Let $\uS^{-1}\uA$ be the Mackey functor obtained by inverting the class $[G/e] \in \uA(G/G)$. 
	\begin{enumerate}[(a)]
		\item There is an isomorphism of Tambara functors $\uS^{-1}\uA \cong \uZ[\tfrac{1}{|G|}]$.
		\item Any $\uS^{-1}\uA$-module $\uM$ is isomorphic to the fixed point functor on $\uM(G/e)$. 
		\item There is an equivalence of categories $\uS^{-1}\uA\mhyphen\Mod \simeq \z[\tfrac{1}{|G|}][G]\mhyphen\Mod$. 
	\end{enumerate}
\end{thmx}

\begin{remark}
	In representation theory, it is common for many situations to be simplified by inverting the order of the group. Inverting $[G/e] \in \uA(G/G)$ is the categorification of this process, where we invert the class of $G$ itself, rather than its image in $\z$. One consequence of \cref{Thmx:LocalBrutal} is that $\uS^{-1}\uA$-modules are cohomological $\uA[\tfrac{1}{|G|}]$-modules, where $\uA[\tfrac{1}{|G|}]$ is the Mackey functor obtained from $\uA$ by inverting $|G| \in \uA(G/G)$. This makes the decategorification precise, in some sense. 
\end{remark}

\subsection{Outline}

In \cref{Sec:IndexingSystems}, we recall indexing categories and categories of polynomials. In particular, we recall the commutative semiring structure of hom-sets in categories of polynomials (\cref{Thm:Tam93}). 

In \cref{Sec:Tambara}, we discuss the main notions relevant for the paper. We first recall Mackey and incomplete Tambara functors and introduce the Burnside Tambara functor (\cref{exm:BurnsideTambara}). We then discuss the box product, modules over incomplete Tambara functors, and the notions of freeness, projectivity, and flatness appearing in our main theorems. We conclude by introducing free incomplete Tambara functors (\cref{Def:FreeIncTamb}). 

In \cref{Sec:Norms}, we recall the norm for Mackey and incomplete Tambara functors from the work of Hoyer \cite{Hoy14} and Blumberg and the first author \cite{BH18}. We apply the norm to reduce \cref{Thmx:Free} to a special case, which then follows from a result about free Green functors from \cite{BH19}.

In \cref{Sec:GFP}, we extend the geometric fixed points functor for Mackey functors (cf. \cite{BGHL19}) to incomplete Tambara functors (\cref{Def:GFP}) and show that under certain hypotheses, the geometric fixed points of free Mackey and incomplete Tambara functors are again free (\cref{Thm:GFPFree} and \cref{Thm:GFPFreeIncomplete}). The preservation of freeness by the norm and geometric fixed points are both used crucially in \cref{Sec:Converse}. 

In \cref{Sec:Converse}, we prove \cref{Thmx:Free2} and \cref{Thmx:Asymptotics}. To prove the general part of \cref{Thmx:Free2}, we use the geometric fixed points functor to reduce to known cases. We then prove \cref{Thmx:Asymptotics} by reducing to solvable groups using a group theory result of Camina--Everest--Gagen \cite{CEG86}; the result then follows from bounds obtained from \cref{Thmx:Solvable}. 

In \cref{Sec:FreeAfterLocalization}, we prove our last two main theorems. We begin by recalling cohomological Mackey functors and their connection to fixed point functors. We then recall localization for incomplete Tambara functors and describe certain localizations of the Burnside Tambara functor. Our main result is \cref{prop:invertGiso}, which identifies a localization of the Burnside Tambara functor with a constant Tambara functor. We then study modules over these localizations, building on a classical result of Greenlees--May \cite{GM95} and Th{\'e}vanaz--Webb \cite{TW95}. Altogether, these results assemble into \cref{Thmx:LocalBrutal}. We then apply this analysis to prove \cref{Thmx:FreeLocalUnderlyingFree}. 

In \cref{App:Tables}, we apply \cref{Thmx:Solvable} to describe which free incomplete Tambara functors are free as Mackey functors for some small finite groups. 

\subsection{Conventions}

\begin{enumerate}[(1)]
	\item All groups are finite. Throughout the paper, $G$ is a finite group and $H$ is a subgroup of $G$. 
	\item All rings are commutative and unital. 
	\item We use the symbol $<$ to denote a strict inclusion of subgroups, whereas $\leq$ denotes an inclusion that is not necessarily strict. 
	\item We will write $\co$ for an indexing category.
	\item When an indexing category $\co$ and $\co$-Tambara functor $\uR$ have been fixed, we will suppress the indexing category $\co$ from our terminology. In particular, ``$\uR$-algebra'' means ``$\uR$-algebra in $\co$-Tambara functors''. 
	\item $\uA$ denotes the Burnside Tambara functor, or its image under the forgetful functor to incomplete Tambara functors or Mackey functors. 
	\item A $G$-ring is a ring with an action of the group $G$ by ring homomorphisms. If $R$ is a $G$-ring, we write $\FP(R)$ for its fixed point Tambara functor (\cref{exm:FixedPointFunctor}). If $G$ acts trivially on $R$, we write $\uR$ for $\FP(R)$. 
\end{enumerate}

\subsection{Acknowledgments}

The authors thank Scott Balchin, Andrew Blumberg, Guchuan Li, Vitaly Lorman, Inna Zakharevich, and Mingcong Zeng for helpful discussions. The authors additionally thank Inna Zakharevich for reading early drafts, and Michael Stahlhauer for catching an omission.

\section{Indexing Categories}\label{Sec:IndexingSystems}

\subsection{Indexing categories}

We recall the definition of an indexing category from \cite{BH18}. Fix a finite group $G$. 

\begin{definition}
	Let $\cc$ be a category. We say a subcategory $\cd$ of $\cc$ is: 
	\begin{enumerate}[(a)]
		\item \emph{wide} if it contains all objects; 
		\item \emph{finite coproduct complete} if $\cd$ has all finite coproducts and coproducts are created in $\cc$;
		\item \emph{pullback stable} if $\cc$ admits pullbacks and whenever 
			\[ \begin{tikzcd}
				A 
					\arrow{r} 
					\arrow{d}[left]{f} 
					& 
				B 
					\arrow{d}{g} 
					\\
				C 
					\arrow{r} 
					&
				D
			\end{tikzcd}\]
			is a pullback diagram in $\cc$ with $g \in \cd$, the morphism $f$ is also in $\cd$. 
	\end{enumerate}
\end{definition}

\begin{definition}[{cf. \cite[Section 3]{BH18}}]
	An \textsl{indexing category} $\co$ for $G$ is a wide, pullback stable, finite coproduct complete subcategory of the category $\Set^G$ of finite $G$-sets. 
\end{definition}

Indexing categories for $G$ form a poset under inclusion. This poset is finite whenever $G$ is finite. The least element of this poset is the \textsl{trivial indexing category} $\co^\triv$, and the greatest element is the \textsl{complete indexing category} $\co^\cplt$. 

\begin{definition}
\label{Def:TrivialIndexingCategory}
	The \emph{trivial indexing category} $\co^{\triv}$ is the wide subcategory of $\Set^G$ containing all morphisms $g \colon X \to Y$ that preserve isotropy, i.e. for all $x \in X$, the stabilizer of $g(x)$ is also the stabilizer of $x$. 

	We also define the \emph{complete indexing category} $\co^{\cplt} := \Set^G$.
\end{definition}

\begin{example}\label{Ex:IndexingCategoriesCp}
	In this example we describe the poset of indexing categories for $C_p$. If an indexing category $\co$ for $C_p$ contains the projection $C_p/e \to C_p/C_p$, then it contains all morphisms of finite transitive $C_p$-sets and, because it is finite coproduct complete, contains all morphisms of finite $C_p$-sets. In this case, the indexing category is the complete indexing category $\co^\cplt = \Set^{C_p}$. 

	If an indexing category does not contain the morphism $C_p/e \to C_p/C_p$, then the only morphisms on transitive $C_p$-sets are the identities. Under finite coproducts, this yields all morphisms of finite $C_p$-sets that preserve isotropy, but no more. Thus, the indexing category must be the trivial indexing category $\co^\triv$. 
\end{example}

We can also vary the group: if \(H\) is a subgroup of \(G\), then for any indexing category \(\co\) for \(G\), there is an associated indexing category \(i_H^\ast\co\).
 
\begin{proposition}[{\cite[Proposition 6.3]{BH18}}]\label{Prop:RestrictIndexCat}
	Let $i_H^* \colon \Set^G \to \Set^H$ be the restriction functor. If $\co$ is an indexing category for $G$, the image of the restriction of \(i_H^\ast\) to \(\co\) lands in \(i_H^\ast\co\).
\end{proposition}

Because we will frequently use it later, we record one more definition here. 

\begin{definition}
	Let $\co$ be an indexing category for $G$. We say that $H/K$ is an \emph{admissible $H$-set for $\co$} if $\co$ contains a morphism $G/K \to G/H$. 
\end{definition}

\begin{example}
	In \cref{Ex:IndexingCategoriesCp}, $C_p/e$ is an admissible $C_p$-set for $\co^\cplt$ but not for $\co^\triv$. 
\end{example}

\begin{remark}
	The classification of indexing categories is an open problem for most groups, but has been completed for cyclic groups of prime power order by Balchin--Barnes--Roitzheim \cite{BBR19}. Indexing categories for other finite groups are analyzed in \cite{BBPR20, Rub20}. 
\end{remark}

\subsection{Categories of polynomials}

To each indexing category $\co$, we associate a category of polynomials with exponents in $\co$. These categories are the domains of incomplete Tambara functors. 

\begin{definition}
	Let $\cd$ be a wide, pullback stable subcategory of $\Set^G$. Let $\cp_\cd^G$ denote the \emph{category of polynomials with exponents} in $\cd$. Objects are finite $G$-sets and  morphisms $\cp_\cd^G(X,Y)$ are \emph{polynomials with exponents in $\cd$}, i.e. equivalence classes of diagrams 
	\[ [X \xleftarrow f A \xrightarrow g B \xrightarrow h Y] \]
	with $g \in \mathcal \cd$. Two such diagrams are equivalent if there is a diagram of the form 
	\[ \begin{tikzcd}[row sep=small]
			& 
		A 
			\arrow{dl}[above]{f} 
			\arrow{r}{g}
			\arrow{dd}{\cong}
			& 
		B 
			\arrow{dd}{\cong}
			\arrow{dr}{h}
			\\
		X 
			& 
			& 
			& 
		Y 
			\\
			& 
		A' 
			\arrow{r}[below]{g'}
			\arrow{ul}{f'}
			& 
		B' 
			\arrow{ur}[below]{h'}
	\end{tikzcd}\]
	Composition of polynomials is given by \cite[Proposition 7.1]{Tam93}.
\end{definition}

Below, we will describe how to work with the morphisms in this category in a more practical way. First, we give a few important examples. 

\begin{example}\label{Example:SetGIso}
	Any wide, pullback stable subcategory of $\Set^G$ contains the subcategory $\Set^G_{\Iso}$ of finite $G$-sets and isomorphisms. This subcategory is wide and pullback stable, so yields a category of polynomials $\cp^G_{\Iso}$. Any polynomial of the form $X \leftarrow A \cong B \to Y$ is canonically isomorphic to one of the form $X \leftarrow A \xrightarrow\id A \to Y$, and thus the category $\cp^G_{\Iso}$ is isomorphic to the category of spans of finite $G$-sets. The coproduct completion of $\cp^G_{\Iso}$ is the \emph{Burnside category of $G$}, denoted $\ca^G$. 
\end{example}

\begin{example}
	If $\co$ is any indexing category, then we get a category $\cp^G_\co$ of polynomials with exponents in $\co$. 
\end{example}

The fact that $\cp_\cd^G$ is a category is not obvious, and composition can be messy. We define a generating set of morphisms and describe how to compose them following \cite{BH18}. 

\begin{definition}
	Let $f \colon X \to Y$ be a morphism of finite $G$-sets. Define three morphisms in $\cp_\co^G(X,Y)$
	\[ R_f := [X \xleftarrow{f} Y \xrightarrow\id Y \xrightarrow\id Y ] \]
	\[ N_f := [X \xleftarrow\id X \xrightarrow{f} Y \xrightarrow\id Y ] \]
	\[ T_f := [X \xleftarrow\id X \xrightarrow\id X \xrightarrow{f} Y ] \]
\end{definition}

\begin{theorem}[{cf. \cite[Section 2.1]{BH18}}]\label{thm:StructureOfPGO}
	\mbox{}
	\begin{enumerate}[(a)]
		\item $R, N, T$ give functors from $\Set^G$ to $\cp_{\co}^G$. $R$ is contravariant; $N$ and $T$ are covariant.  
		\item Any morphism in $\cp_\co^G$ can be written as a composite 
			\[ T_h \circ N_g \circ R_f = [X \xleftarrow{f} A \xrightarrow{g} B \xrightarrow{h} Y]. \]
		\item Given a pullback diagram of finite $G$-sets 
			\[ \begin{tikzcd}
				X' 
					\arrow{r}{f'}
					\arrow{d}{g'}
					& 
				Y' 
					\arrow{d}{g}
					\\
				X 
					\arrow{r}{f}
					& 
				Y 
			\end{tikzcd}\]
			then we have 
			\begin{align*}
				R_g \circ N_f &= N_{f'} \circ R_{g'}, \\
				R_g \circ T_f &= T_{f'} \circ R_{g'}. 
			\end{align*}
		\item Given any diagram isomorphic to one of the form (an \emph{exponential diagram}) 
			\[\begin{tikzcd}
				S 
					\arrow{d}{h}
					& 
				A 
					\arrow{l}[above]{g}
					&  
				S \times_T \Pi_h A 
					\arrow{l}[above]{f'}
					\arrow{d}{g'}
					\\
				T 
					& 
					& 
				\Pi_h A 
					\arrow{ll}{h'}
			\end{tikzcd}\]
			then  
			\[N_{h} \circ T_g = T_{h'} \circ N_{g'} \circ R_{f'}. \]
	\end{enumerate}
	Here, $\Pi_h A$ is the \emph{dependent product}, right adjoint to the pullback $h^* \colon \Set^G_{/T} \to \Set^G_{/S}$.
\end{theorem}

There is more structure on the morphism sets of $\cp^G_\co$. 

\begin{theorem}[{\cite[Proposition 7.6]{Tam93}}]\label{Thm:Tam93} 
	$(\cp_\co^G(X,Y), +, \cdot, 0, 1)$ is a commutative semiring, with units
	\begin{align*}
		0 &= [X \leftarrow \emptyset \rightarrow \emptyset \rightarrow Y], \\
		1 &= [X \leftarrow \emptyset \rightarrow Y \xrightarrow{\id} Y]. 
	\end{align*}
	The operations are defined on polynomials $\Sigma = {[X \xleftarrow f A \xrightarrow g B \xrightarrow h Y]}$ and $\Sigma' = {[X \xleftarrow{f'} A' \xrightarrow{g'} B' \xrightarrow{h'} Y]}$ by 
	\begin{align*}
		\Sigma + \Sigma' &= \left[X \xleftarrow{\nabla \circ (f \sqcup f')} A \sqcup A' \xrightarrow{g \sqcup g'} B \sqcup B' \xrightarrow{\nabla \circ (h \sqcup h')} Y\right], \text{and}  \\
		\Sigma \cdot \Sigma' &= 
			\left[ 
				X \xleftarrow{\hat f} (A \times_Y B') \sqcup (B \times_Y A') \xrightarrow{\hat g} 
				B \times_Y B' \xrightarrow{\hat h} Y
			\right],
	\end{align*}
	where:
	\begin{itemize}
		\item $\nabla \colon S \sqcup S \to S$ is the codiagonal (fold) morphism;
		\item $\hat f$ is the morphism given by projecting onto $A$ or $A'$ then applying $f$ or $f'$; 
		\item $\hat g$ is the morphism given by applying $g$ to $A$ or $g'$ to $A'$;
		\item $\hat h$ is the morphism given by either of the two equivalent morphisms $h \circ \pr_1$ or $h' \circ \pr_2$.  
	\end{itemize}
\end{theorem}

\begin{remark}
	Note that, although $\cp^G_\co(X,Y)$ is a commutative semiring, the additive completion of $\cp^G_\co$ is \emph{not} enriched over rings. This fails because the multiplicative unit 
	\[1 = [X \leftarrow \emptyset \to X \xrightarrow\id X]\]
	in $\cp^G_\co(X,X)$ is not the identity morphism 
	\[\id_X = [X \leftarrow X \to X \to X].\] 
\end{remark}

\begin{remark}\label{Rmk:WeylAction}
	The Weyl group $W_GK := N_G K/K$ acts on $G/K$ by conjugation, so each $\gamma \in W_GK$ defines a $G$-equivariant map $c_\gamma: G/K \to G/K$. This extends to a $W_GK$-action on $\cp^G_{\co}(G/H,G/K)$ by transferring along $c_\gamma$, i.e.
	\[\gamma [G/H \leftarrow A \rightarrow B \xrightarrow{h} G/K] = [G/H \leftarrow A \rightarrow B \xrightarrow{c_\gamma \circ h} G/K].\]
\end{remark}

\section{Incomplete Tambara Functors}\label{Sec:Tambara}

\subsection{Mackey and incomplete Tambara functors}\label{SS:MTF}

In equivariant algebra, Mackey functors are the analog of abelian groups and incomplete Tambara functors are the analog of commutative rings. We recall these notions along with some special examples. 

\begin{definition}
	The \emph{Burnside category} $\ca^G$ is the group completion of the category of spans of finite $G$-sets. Objects are finite $G$-sets, and each hom-set $\ca^G(X,Y)$ is the group completion of the hom-set of spans from $X$ to $Y$, which is a monoid under disjoint union. 
\end{definition}

The Burnside category $\ca^G$ is an additive category with direct sums given by disjoint union of finite $G$-sets. 

\begin{definition}
	A \emph{Mackey functor} is an additive functor $\uM \colon \ca^G \to \Ab$. A \textsl{morphism of Mackey functors} is a natural transformation. We write $\Mack_G$ for the category of Mackey functors for the group $G$. 
\end{definition}

We give another, equivalent, definition of Mackey functors that emphasizes the connection with incomplete Tambara functors defined below.

\begin{proposition}[{\cite[Proposition 4.3]{BH18}}]\label{Def:Mackey}
	A Mackey functor for $G$ is a product-preserving functor $\uM \colon \cp^G_{\Iso} \to \Set$ such that $\uM(X)$ is an abelian group for every finite $G$-set $X$. 
\end{proposition}

We now define the analog of commutative rings in equivariant algebra.

\begin{definition}[{\cite[Definition 4.1]{BH15}}]\label{Def:OTambara}
	Let $\co$ be an indexing category. An \emph{$\co$-Tambara functor} for $G$ is a product-preserving functor $\uR \colon \cp_{\co}^G \to \Set$ such that $\uR(X)$ is an abelian group for every finite $G$-set $X$. 
	An \emph{incomplete Tambara functor} is an $\co$-Tambara functor for some $\co$. 

	A \emph{morphism of $\co$-Tambara functors} is a natural transformation. We write $\co\mhyphen\Tamb_G$ for the category of $\co$-Tambara functors for the group $G$.
\end{definition}

Because any Tambara functor $\uR$ is product-preserving, and disjoint union is the categorical product in $\cp^G_\co$, it suffices to give the value of $\uR$ on each transitive finite $G$-set $G/K$. Moreover, because any morphism in the category $\cp_{\co}^G$ may be written as a composite $T_h \circ N_g \circ R_f$, it suffices to define the Tambara functor on morphisms of the form $T_h$, $N_g$, and $R_f$. 

\begin{definition}
\label{Def:NormTransferRestriction}
	We write $\tr_h := \uR(T_h)$, $\nm_g := \uR(N_g)$, and $\res_f := \uR(R_f)$. These are the \emph{transfer}, \emph{norm}, and \emph{restriction} of $\uR$, respectively. In the case when $\pi \colon G/K \to G/H$ is the canonical projection, we write $\tr_K^H := \tr_\pi$, $\nm_K^H := \nm_\pi$ and $\res_K^H := \res_\pi$.
\end{definition}

The data of an $\co$-Tambara functor is equivalent to the following:
\begin{itemize}
	\item A collection of commutative $W_GH$-rings $\uR(G/H)$, one for each subgroup $H$ of $G$;
	\item restriction homomorphisms (of non-unital commutative rings) 
		\[\res^H_K \colon \uR(G/H) \to \uR(G/K);\]
	\item transfer homomorphisms (of abelian groups) 
		\[\tr_K^H \colon \uR(G/K) \to \uR(G/H); \]
	\item and multiplicative norm morphisms 
		\[\nm_K^H \colon \uR(G/K) \to \uR(G/H)\] whenever $H/K$ is an admissible $H$-set for $\co$.
\end{itemize}
These data are subject to some conditions \cite[Section 2]{Tam93}.

\begin{example}[{\cite[Section 4]{BH18}}]
	There are two special classes of incomplete Tambara functors which appear in equivariant algebra:

	\begin{enumerate}[(a)]

		\item An $\co^{\triv}$-Tambara functor is a \emph{Green functor.}

		\item An $\co^{\cplt}$-Tambara functor is a \emph{Tambara functor}. 
	\end{enumerate}
\end{example}

\begin{proposition}[{\cite[Corollary 4.4]{BH18}}]
	Every $\co$-Tambara functor $\uR \colon \cp^G_{\co} \to \Set$ has an underlying Mackey functor given by restricting the domain of $\uR$ to the category $\cp^G_{\Iso}$. 
\end{proposition}

\begin{proposition}[{\cite[Proposition 5.14]{BH18}}]
	Let $\co$ and $\co'$ be indexing categories such that $\co'$ is a subcategory of $\co$. Then every $\co$-Tambara functor $\uR \colon \cp^G_{\co} \to \Set$ has an underlying $\co'$-Tambara functor given by restricting the domain of $\uR$ to the category $\cp^G_{\co'}$. 
\end{proposition}

In particular, any $\co$-Tambara functor has an underlying Green functor by restricting the domain to $\cp^G_{\co^\triv}$. 

There are several fundamental Tambara functors we will use in this paper. Each can be regarded as an incomplete Tambara functor by forgetting structure.

\begin{definition}[{\cite[Example 3.2]{Tam93}}]\label{exm:BurnsideTambara}
	The \emph{Burnside Tambara functor} is the Tambara functor $\uA$ whose value on a $G$-set $X$ is the additive group completion of the semiring of isomorphism classes of $G$-sets $Y$ over $X$, with semiring operations given by Cartesian product and disjoint union. 

	When $X = G/H$, $\uA(G/H)$ is the group completion of the monoid of finite $H$-sets under disjoint union. The restriction $\res_K^H$ is given by pullback, transfer $\tr_K^H$ is induction, and norm $\nm_K^H$ is coinduction. 
\end{definition}

\begin{remark}\label{Rmk:BurnsideInitial}
	The Burnside Tambara functor is the initial incomplete Tambara functor, just like $\z$ is the initial commutative ring. Modules over $\uA$, or Mackey functors, are the equivariant analog of abelian groups, while algebras over $\uA$ are incomplete Tambara functors.
\end{remark}

\begin{example}\label{exm:C2BurnsideTambara}
	We examine the structure of the $C_p$-Burnside functor. 

	At the underlying level, $\uA(C_p/e)$ is generated by isomorphism classes of finite $C_p$-sets over $C_p/e$. Such a $C_p$-set must be a disjoint union of orbits isomorphic to $C_p/e$. Hence, 
	\[ \uA(C_p/e) \cong \z \]

	At level $\uA(C_p/C_p)$, generators may be any finite $C_p$-set since $C_p/C_p = *$ is terminal. Any $C_p$-set is a disjoint union of the transitive $C_p$-sets, $C_p/e$ and $C_p/C_p$. Let $t$ be the class of $C_p/e$. We have a relation 
	\[t \cdot t = [C_p/e \times C_p/e] = \left[\bigsqcup_{i=1}^p C_p/e \right] = p\,t.\]
	Therefore, 
	\[ \uA(C_p/C_p) \cong \z[t] / (t^2 - pt) \]

	The restriction map sends the class of $C_p/e$ to its underlying finite set, and is therefore determined by $t \mapsto p$. The transfer is induction and is given by multiplication by $t$. When $\uA$ is considered as an incomplete Tambara functor, norms that exist are coinduction and are given by \cite[Example 1.4.6]{Maz13}:
	\[ \nm_e^{C_p}(a) = a + \left(\frac{a^p - a}{p}\right) t. \]
\end{example}

\begin{example}
    For \(G=C_p\), with \(p\) a prime, we consider the Mackey functor defined by
    \[
    \uZ(G/H)=\z,
    \]
    with the restriction map given by the identity, the transfer map given by multiplication by \(p\), and the Weyl group acting trivially. This is the \emph{constant Mackey functor \(\uZ\)}, written \(\uZ\).
\end{example}

\begin{example}
    For \(G=C_p\), with \(p\) a prime, we consider the Mackey functor defined by
    \[
    \uZ^{\ast}(G/H)=\z,
    \]
    with transfer map given by the identity, the restriction map given by multiplication by \(p\), and the Weyl group acting trivially. This is the \emph{dual to the constant Mackey functor \(\uZ\)}, written \(\uZ^{\ast}\).
\end{example}

\subsection{Modules over incomplete Tambara functors}\label{SS:Modules}

\begin{definition}[{\cite{Lew81}}]
	Let $\uM$ and $\uN$ be Mackey functors.
	The \emph{box product} $\uM \boxtimes \uN$ is the Mackey functor obtained by left Kan extending the tensor product of abelian groups along the functor 
		\[\times \colon \cp^G_{\Iso} \times \cp^G_{\Iso} \to \cp^G_{\Iso} \]
	given on objects by  $(T,T') \mapsto T \times T'$. 
	\[ \begin{tikzcd}
		\cp^G_{\Iso} \times \cp^G_{\Iso}  
			\arrow{r}{\uM \times \uN}
			\arrow{d}[left]{\times}
			& 
		\Ab \times \Ab 
			\arrow{r}{\otimes_\z}
			& 
		\Ab 
			\\
		\cp^G_{\Iso}  
			\arrow[dashed]{urr}[below]{\uM \boxtimes \uN}
	\end{tikzcd}\]
\end{definition}

\begin{theorem}[{\cite{Lew81}}] 
	\mbox{}
	\begin{enumerate}[(a)]
		\item The box product makes the category of Mackey functors into a symmetric monoidal category with unit the Burnside Mackey functor. 
		\item Green functors are monoids for this symmetric monoidal structure on the category of Mackey functors. 
	\end{enumerate}
\end{theorem}

The box product also plays a special role for Tambara functors.

\begin{theorem}[{\cite{Strickland12,BH18}}]
	If $\uR$ and $\uR'$ are $\co$-Tambara functors, then $\uR \boxtimes \uR'$ has a natural structure as an $\co$-Tambara functor.

	The natural maps \(\uR\to \uR\boxtimes\uR'\leftarrow \uR'\) are maps of \(\co\)-Tambara functors and witness the box product as the coproduct.
\end{theorem}

\begin{definition}
	If $\uR$ is a Green functor with unit $\eta : \uA \to \uR$ and multiplication $\mu : \uR \boxtimes \uR \to \uR$, then an \emph{$\uR$-module $\uM$} is a Mackey functor $\uM$ together with a homomorphism of Mackey functors $\nu : \uR \boxtimes \uM \to \uM$ such that the the following diagram commutes:
	\[
	\begin{tikzcd}
		\uR \boxtimes \uR \boxtimes \uM 
			\arrow{r}{\mu \boxtimes 1} 
			\arrow{d}{1 \boxtimes \nu} 
			& 
		\uR \boxtimes \uM 
			\arrow{d}{\nu} 
			& 
		\uM 
			\arrow{l}[above]{\eta \boxtimes 1} 
			\arrow{dl}{1} 
			\\
		\uR \boxtimes \uM 
			\arrow{r}{\nu} 
			& 
		\uM.
	\end{tikzcd}
	\]
	A \emph{morphism of $\uR$-modules} $f \colon \uM \to \uN$ is a homomorphism of Mackey functors such that the following diagram commutes: 
	\[ \begin{tikzcd}
		\uR \boxtimes \uM 
			\arrow{r}{\nu_{\uM}}
			\arrow{d}{\id \boxtimes f}
			& 
		\uM 
			\arrow{d}{f}
			\\
		\uR \boxtimes \uN 
			\arrow{r}{\nu_{\uN}}
			& 
		\uN. 
	\end{tikzcd}\]

	If $\uR$ is an $\co$-Tambara functor, an \emph{$\uR$-module $\uM$} is a module over the underlying Green functor of $\uR$. Likewise with right $\uR$-modules.

	Denote the category of $\uR$-modules by $\uR\mhyphen\Mod$. 
\end{definition}

\begin{definition}
	If $\uM$ and $\uN$ are $\uR$-modules, we define the \emph{relative box product:}
	\[
	\begin{tikzcd}
		\uM \boxtimes_{\uR} \uN := \coeq\big( \uM \boxtimes \uR \boxtimes \uN 
			\arrow[yshift=2]{rr}{\nu_{\uM} \boxtimes \id_{\uN}} 
			\arrow[yshift=-2]{rr}[below]{\id \boxtimes \nu_{\uN}} 
			& 
			&
		\uM \boxtimes \uN\big).
	\end{tikzcd}
	\]
\end{definition}

\begin{proposition}[{\cite{Lew81}}]
	\mbox{} 
	\begin{enumerate}[(a)]
		\item The relative box product $\boxtimes_{\uR}$ makes the category $\uR\mhyphen\Mod$ into a symmetric monoidal category with unit $\uR$.
		\item $\uR\mhyphen\Mod$ is an abelian category.
	\end{enumerate}
\end{proposition}

\begin{remark}
	This notion of module is the extension of Strickland's notion of \emph{naive module} \cite[Definition 14.1]{Strickland12} to incomplete Tambara functors. Strickland points out that this definition only uses the underlying Green functor structure of $\uR$ and proposes a new notion of module in \cite[Definition 14.3]{Strickland12} which incorporates extra Tambara structure. The first author showed in \cite{Hil17} that Strickland's genuine modules are the abelian group objects in the category of $\uR$-Tambara functors. We will not consider these more refined notions of modules here and refer the reader to \cite{Hil17} for further discussion. 
\end{remark}

\begin{definition}
	Let $T$ be a finite $G$-set. The \emph{free Mackey functor on a generator at level $T$}, denoted $\uA\{x_{T}\}$, is the Mackey functor defined by
	\[\uA\{x_{T}\} := \cp^G_{\text{Iso}}(T,-)^+,\]
	where the superscript $+$ denotes group completion.

	Let $\uR$ be an incomplete Tambara functor. The \emph{free $\uR$-module on a generator at level $T$}, denoted $\uR\{x_{T}\}$, is defined by
	\[\uR\{x_{T}\} := \uR \boxtimes \uA\{x_{T}\}.\]
\end{definition}

\begin{example}
	The free $\uR$-module on a single generator at level $G/G$ is $\uR$ itself: 
	\[\uR\{x_{G/G}\} = \uR.\]
\end{example}

\begin{definition}
	A \emph{free $\uR$-module} is any $\uR$-module $\uM$ with
	\[\uM \cong \bigoplus_{i \in I} \uR\{x_{T_i}\}\]
	where each $T_i$ is a finite $G$-set. 
\end{definition}

\begin{remark}
	As $\uR\{x_{T}\} \oplus \uR \{x_{T'}\} \cong \uR\{x_{T \sqcup T'}\}$, any finitely generated free $\uR$-module has the form $\uR \{ x_U\}$ for a single finite $G$-set $U$. Importantly, this is not the case if the free $\uR$-module is infinitely generated. Nevertheless, if a free $\uR$-module is infinitely generated, it is a direct limit of finitely generated free $\uR$-modules. 
\end{remark}

Free modules enjoy the following universal property. 

\begin{lemma}
	Free $\uR$-modules on a single generator represent evaluation:
	\[\uR \mhyphen\Mod(\uR\{x_T\}, \uM) \cong \uM(T).\]
\end{lemma}

\begin{definition}
	An $\uR$-module $\uM$ is \emph{projective} if it is a summand of a free $\uR$-module. 

	An $\uR$-module $\uM$ is \emph{flat} if the functor $- \boxtimes_{\uR} \uM : \uR \mhyphen \Mod \to \uR \mhyphen \Mod$ is exact. 
\end{definition}

Free, projective, and flat modules interact in the way one would expect from classical homological algebra. 

\begin{proposition}
	Projective $\uR$-modules are flat. 
\end{proposition}

\begin{proposition}\label{prop:base-change-of-free-is-free}
	Let $f \colon \uR \to \uS$ be a morphism of $\co$-Tambara functors. The base change functor 
	\[ \uS \boxtimes_{\uR} (-) \colon \uR\mhyphen\Mod \to \uS\mhyphen\Mod \]
	takes free (resp. projective resp. flat) $\uR$-modules to free (resp. projective resp. flat) $\uS$-modules.
\end{proposition}

\begin{proof}
	The functor $\uS \boxtimes_{\uR} (-)$ commutes with direct sums and direct limits, so $\uS \boxtimes_{\uR} (-)$ takes free $\uR$-modules to free $\uS$-modules. 

	Alternatively, let $\uM$ be an $\uS$-module. We have
	\[\uS\mhyphen\Mod(\uS \boxtimes_{\uR} \uR\{x_T\}, \uM) \cong \uR\mhyphen\Mod(\uR\{x_T\}, \uM) \cong \uM(T),\]
	so $\uS \boxtimes_{\uR} \uR\{x_T\}$ represents evaluation at $T$. 

	If $\uP$ is projective over $\uR$, then there is some $\ul Q$ such that $\uP \oplus \ul Q$ is free over $\uR$. Then $(\uS\boxtimes_{\uR}\uP) \oplus (\uS \boxtimes_{\uR} \ul Q)$ is free over $\uS$, and $\uS \boxtimes_{\uR} \uP$ is projective over $\uS$. 

	If $\uM$ is flat over $\uR$, then $(\uS \boxtimes_{\uR} \uM) \boxtimes_{\uS} (-)$ is naturally isomorphic to $\uM \boxtimes_{\uR} (-)$, and therefore exact. 
\end{proof}

Flatness is a surprisingly subtle concept for Mackey functors. The following example will be used often.

\begin{lemma}\label{Lem:ZDualNotFlat}
    For \(G=C_p\) with \(p\) a prime, neither \(\uZ\) nor \(\uZ^{\ast}\) is flat.
\end{lemma}
\begin{proof}
    We have two exact sequences
    \[
    0\to \uZ^{\ast}\to \uA\xrightarrow{p-t} \uA \to \uZ\to 0
    \]
    and 
    \[
    0\to \uZ\to \uA\{x_{C_p}\}\xrightarrow{1-\gamma} \uA\{x_{C_p}\}\to \uZ^\ast\to 0,
    \]
    where \(\gamma\) is a generator of \(C_p\). Splicing these together gives a projective (even free) resolution of either \(\uZ\) or \(\uZ^{\ast}\). 
    
    If we consider the augmentation ideal \(\uI\), the kernel of the map \(\uA\to\uZ\), then this has the property that the restriction to the trivial group is zero. This means that \(\uI\boxtimes \uA\{x_{C_p}\}=0\). This together with the projective resolution shows that while \(\uZ^\ast\boxtimes \uI\) is zero, \(\uTor^3(\uZ^\ast, \uI) \cong \uI \) is not.
\end{proof}

\subsection{Free incomplete Tambara functors}\label{SS:Algebras}

We now introduce free incomplete Tambara functors, which are the main object of study in this work.

\begin{definition}[{\cite[Definition 5.4]{BH18}}]\label{Def:FreeIncTamb}\label{def:FreeIncompleteTambaraFunctor}
	For a finite $G$-set $T$, define 
	\[\uA^{\co}[x_{T}] := \cp^G_{\co}(T, -)^+\]
	be the group completion of the functor represented by $T$. We will refer to this as the \emph{free $\co$-Tambara functor on a generator at level $T$}. 

	More generally, if $\uR$ is an $\co$-Tambara functor, the \emph{free $\uR$-algebra on a generator at level $T$}, denoted $\uR[x_{T}]$, is defined by
	\[\uR[x_{T}] := \uR \boxtimes \uA^{\co}[x_{T}].\]
\end{definition}

\begin{remark}\label{rem:RepEval}
	The free $\co$-Tambara functor $\uA^{\co}[x_{T}]$ represents evaluation at $T$:
	\[\co\mhyphen\Tamb(\uA^{\co}[x_{T}], \uR) \cong \uR(T).\]
\end{remark}

The norm, transfer, and restriction of a free incomplete Tambara functor are given by post-composition, i.e. applying $T_\phi \circ -$, $N_\phi \circ -$, or $R_\psi \circ -$ to a polynomial $\Sigma = T_h \circ N_g \circ R_f$. For the transfer and restriction, simple descriptions of this operation are possible. 

\begin{proposition}\label{Cor:ResTrFormulas}
	Let
	\[\Sigma = [G/H \xleftarrow{f} A \xrightarrow{g} B \xrightarrow{h} Y] \in \uA^\co[x_{G/H}](Y) = \cp^G_{\co}(G/H,Y)^+\]
	be a generator of $\uA^\co[x_{G/H}](Y)$. If $j : Y \to Z$ is a map of $G$-sets, then
	\[\tr_j(\Sigma) = [G/H \xleftarrow{f} A \xrightarrow{g} B \xrightarrow{j \circ h} Z].\]
	If $k : Z \to Y$ is a function of $G$-sets, then
	\[\res_k(\Sigma) = [G/H \xleftarrow{f \circ \pi_A} A \times_Y Z \xrightarrow{g \times \id} B \times_Y Z \xrightarrow{\pi_Z} Z].\]
\end{proposition}

\begin{proof}
	The transfer along $j$ is pre-composition with $T_j$:
	\[\tr_j(\Sigma) = T_j \circ \Sigma = T_j \circ (T_h \circ N_g \circ R_f) = T_{j \circ h} \circ N_g \circ R_f. \]

	Restriction along $k$ is pre-composition with $R_k$: 
	\[ \res_k(\Sigma) = R_k \circ (T_h \circ N_g \circ R_f). \]
	Using \cref{thm:StructureOfPGO}, we may commute $R_k$ with $T_h$ and $N_g$ in turn. The diagram 
	\[ \begin{tikzcd}
		B \times_Y Z 
			\arrow{r}{\pi_Z}
			\arrow{d}{\pi_B}
			& 
		Z 
			\arrow{d}{k}
			\\
		B 
			\arrow{r}{h}
			& 
		Y
	\end{tikzcd}\]
	demonstrates that $R_k \circ T_h = T_{\pi_Z} \circ R_{\pi_B}$, and the diagram 
	\[ \begin{tikzcd}
		A \times_Y Z 
			\arrow{r}{g \times \id}
			\arrow{d}{\pi_A}
			& 
		B \times_Y Z 
			\arrow{d}{\pi_B}
			\\
		A 
			\arrow{r}{g}
			& 
		B
	\end{tikzcd}\]
	demonstrates that $R_{\pi_B} \circ N_g = N_{g \times \id} \circ R_{\pi_A}$. Therefore, 
	\begin{align*}
		\res_k(\Sigma) & = R_k \circ (T_h \circ N_g \circ R_f)\\
	 	& = T_{\pi_Z} \circ R_{\pi_B} \circ N_g \circ R_f \\
		& = T_{\pi_Z} \circ N_{g \times \id} \circ R_{\pi_A} \circ R_f \\
		& = T_{\pi_Z} \circ N_{g \times \id} \circ R_{f \circ \pi_A}. 
	\end{align*}
\end{proof}

\begin{corollary}\label{cor:PolynomialsAsPolynomials}
    For any orbits \(G/H\) and \(G/K\), we have that
    \[
        \cp^G_{\co}(G/H,G/K)^+
    \]
    is the free abelian group with basis given by isomorphism classes of polynomials of the form
    \[
    G/H\xleftarrow{f} S\xrightarrow{g} G/J\xrightarrow{h} G/K
    \]
\end{corollary}

Note that the map \(h\) in \cref{cor:PolynomialsAsPolynomials} is recording the transfer from \(G/J\) to \(G/K\) applied to the polynomial
\[
G/H\xleftarrow{f} S\xrightarrow{g} G/J\xrightarrow{=}G/J.
\]
The transfer maps may identify diagrams of this form, since in general, there are more isomorphisms of diagrams
\[
G/H\xleftarrow{f} S\xrightarrow{g} G/J\xrightarrow{h} G/K.
\]
Note however, that this simply takes basis vectors to other basis vectors.

\section{Norms and Sufficiency}\label{Sec:Norms}

For any subgroup \(H\), the \(G\)-set \(G/H\) is in the image of the induction functor on \(H\)-sets, and this functor extends to polynomials with exponents in various indexing categories. Precomposition with the induction functor 
\[
    G\times_H(-)\colon \cp_{i_H^\ast\co}^{H}\to \cp_{\co}^{G}
\]
is a model for the restriction functor
\[
	i_H^* \colon \co\mhyphen\Tamb_G\to i_H^\ast\co\mhyphen\Tamb_H.
\]
This functor has a left-adjoint, given by left Kan extension.

\begin{definition}[{\cite[Definition 6.8]{BH18}}]
	Let
	\[
		n_H^G \colon i^*_H \co\mhyphen\Tamb_H \to \co\mhyphen\Tamb_G 
	\]
	be the left adjoint to the restriction $i_H^*$ to \(H\).
\end{definition}

\begin{proposition}\label{Prop:NormO}
	There is an isomorphism of $\co$-Tambara functors
	\[
		n_H^G \uA^{i^*_H \co}[x_{H/H}] \cong \uA^{\co}[x_{G/H}].
	\]
\end{proposition}

\begin{proof}
	It suffices to show that $n_H^G \uA^{i^*_H \co}[x_{H/H}]$ represents evaluation at $G/H$ in $\co$-$\Tamb_G$. We compute:
	\begin{align*}
		\co\mhyphen\Tamb_G\big(n_H^G \uA^{i^*_H \co}[x_{H/H}], \uR\big) &\cong i^*_H\co\mhyphen\Tamb_H\big(\uA^{i^*_H\co}[x_{H/H}], i^*_H \uR\big) \\
		&\cong i_H^* \uR(H/H) \\
		&\cong \uR(G/H). 
	\end{align*}
\end{proof}

\begin{example}
	When $H = e$, we have $n_e^G(\z[x]) \cong \uA^{\co^\cplt}[x_{G/e}]$. 
\end{example}

When \(G/H\) is an admissible \(G\)-set, then the underlying Mackey functor for \(n_H^G\uR\) can be computed as a functor of the underlying Mackey functor of \(\uR\). This is a key step in forming the ``external'' or ``\(G\)-symmetric monoidal'' description of Tambara functors.

\begin{definition}[{\cite[Section 2.3]{Hoy14}}]
	The \emph{norm}
	\[N_H^G: \Mack_H \to \Mack_G,\]
	is defined by left Kan extension along the coinduction functor $\Set^H(G,-) \colon \cp^H_{\Iso} \to \cp^G_{\Iso}$. 
\end{definition}

\begin{proposition}\label{Prop:NormFree}
	The norm
	\[N_H^G: \Mack_H \to \Mack_G\]
	takes free $H$-Mackey functors to free $G$-Mackey functors. 
\end{proposition}

\begin{proof}
	If a free $H$-Mackey functor $\uA\{x_T\}$ is finitely generated, we have
	\[N_{H}^G \uA\{x_{T}\} \cong \uA\{x_{\Set^H(G,T)}\}\]
	by definition. If it is not finitely generated, it can be written as a direct limit of finitely generated ones. The norm commutes with direct limits, so again the norm is free. 
\end{proof}

\begin{proposition}[{\cite[Theorem 2.3.3]{Hoy14}, \cite[Theorem 6.15]{BH18}}]
    If \(G/H\) is an admissible \(G\)-set for \(\co\), then we have a natural isomorphism of functors from \(i_H^\ast\co\)-Tambara functors to \(G\)-Mackey functors:
    \[
    U\circ n_H^G\cong N_H^G\circ U.
    \]
\end{proposition}

\begin{corollary}
    If \(G/H\) is an admissible \(G\)-set for \(\co\), then we have a natural isomorphism of Mackey functors
    \[
        U\big(\uA^{\co}[x_{G/H}]\big)\cong N_H^G\big(U\big(\uA^{i_H^\ast\co}[x_{H/H}]\big)\big).
    \]
\end{corollary}

Combined with \cref{Prop:NormFree}, we deduce some guarantees for free underlying Mackey functors.

\begin{corollary}\label{cor:NormsAlongAdmissiblesOfFrees}
    If \(G/H\) is an admissible \(G\)-set for \(\co\), and if the Mackey functor underlying \(\uA^{i_H^\ast\co}[x_{H/H}]\) is free, then the Mackey functor underlying \(\uA^{\co}[x_{G/H}]\) is free.
\end{corollary}

There is one case where we can easily show that the underlying Mackey functor is free: the free Green functor on a fixed generator.

\begin{proposition}\label{Rmk:FreeTrivial}
	The free Green functor on a fixed generator, $\uA^{\co^{\triv}}[x_{G/G}]$, is free as a Mackey functor. 
\end{proposition}

\begin{proof}
	By \cite[Corollary 2.11]{BH19}, there is an isomorphism of Mackey functors
	\[\uA^{\co^{\triv}}[x_{G/G}] \cong \z[x] \otimes \uA,\]
	where $E \otimes \uA$ is the Mackey functor which sends a $G$-set $T$ to $E \otimes \uA(T)$, for an abelian group $E$. This implies that $\uA^{\co^{\triv}}[x_{G/G}]$ is the direct sum of infinitely many copies of $\uA$, so $\uA^{\co^{\triv}}[x_{G/G}]$ is free as a Mackey functor. 
\end{proof}

Applying this to \cref{cor:NormsAlongAdmissiblesOfFrees}, we deduce a class of free Tambara functors that are free as \(\uA\)-modules.

\begin{theorem}\label{Thm:Sufficiency}
	Let $\co$ be an indexing category for a finite group $G$ and let $H$ be a subgroup of $G$. If $i^*_H \co = \co^{\triv}$ and $G/H$ is admissible for $\co$, then $\uA^{\co}[x_{G/H}]$ is free as an $\uA$-module. 
\end{theorem}

\section{Geometric Fixed Points}\label{Sec:GFP}

Studying the converse to \cref{Thm:Sufficiency} requires us to also be able to ``restrict'' along quotient maps. Additively, this is the geometric fixed points functor, but we also need to understand this on Tambara functors.

Throughout this section, $N$ is a normal subgroup of $G$ and $Q := G/N$ is the quotient group. 

\subsection{Cleaving Indexing System}
\begin{notation}
	If \(N\) is a normal subgroup of \(G\), then let
	\[
		\mathcal F_{N}=\{H \leq G \mid N \not\subset H\}
	\]
	be the family of subgroups of \(G\) which do not contain \(N\).
\end{notation}

Associated to \(\mathcal F_N\), we have a universal indexing category.

\begin{definition}\label{defn:FixedN}
	Let \(\co^N_\gen\) be the wide subgraph of \(\Set^G\) such that \(f\colon S \to T\) is in \(\co^N_\gen\) if and only if the canonical map
	\[
		S^N\to S\times_T T^N
	\]
	is an isomorphism.
\end{definition}

We record two useful reformulations of \(\co^N_\gen\).

\begin{proposition}\label{prop:ONInclusion}
The following statements hold:

\begin{enumerate}[(a)]

	\item A map $f \colon S \to T$ is in \(\co^N_\gen\) if and only if the inclusion 
		\[
			S^N\subset f^{-1}(T^N)
		\]
		is the identity.

	\item A map of orbits \(G/K\to G/H\) is in \(\co^N_\gen\) if and only if one of two things hold:
		\begin{enumerate}[(i)]
			\item \(K\) contains \(N\); or
			\item \(H\) does not contain \(N\).
		\end{enumerate}
	
\end{enumerate}
\end{proposition}

\begin{proof}
	Part (a) follows from the observation that the canonical map in \cref{defn:FixedN} is the natural one
	\[
		S^N\to f^{-1}(T^N).
	\]
	Part (b) follows from the identifications
	\[
		(G/H)^N \cong
		\begin{cases}
			\emptyset & N\not\subset H \\
			G/H & N\subset H,
		\end{cases}
	\]
	and part (a). 
\end{proof}

Recall that families of subgroups are the same thing as sieves in the orbit category. Given any normal subgroup \(N\), we have a quotient map \(q\colon G\to Q = G/N\). This gives us a fully-faithful embedding 
\begin{equation}\label{eq:qstaronorbitcategories}
	q^\ast\colon \Orb^Q \hookrightarrow \Orb^G.
\end{equation}
The intersection of \(\co^N_\gen\) with the orbit category of \(G\) is then
\[
	\mathcal F_N\amalg\Orb^Q.
\]

\begin{remark}
	Thinking of an indexing category as parameterizing transfers, we can reinterpret this as saying that we have no transfers from subgroups in the family to those not in the family. We think of this as a chasm, cleaving the groups that contain \(N\) from those that do not. In particular, intersecting any indexing category with this one gives a universal way to remove any transfers that bridge from the family to the complementary cofamily. 
\end{remark}

\begin{theorem}\label{thm:ONIndexingSystem}
	The wide subgraph \(\co^N_\gen\) is an indexing subcategory of \(\Set^G\).
\end{theorem}

\begin{proof}
	We first show that it is a subcategory. It is clear that for any finite \(G\)-set \(T\), the identity map on \(T\) satisfies the conditions of \cref{defn:FixedN}. Now let
	\[
		U\xrightarrow{f} S\xrightarrow{g} T
	\]
	be morphisms in \(\co^N_\gen\). By assumption, we then have
	\[
		U^{N}=f^{-1}(S^{N})=f^{-1}\big(g^{-1}(T^{N})\big)=(g\circ f)^{-1}(T^{N}),
	\]
	as desired. Thus \(\co^N_\gen\) is a subcategory. By assumption, this is a wide subcategory. We apply \cite[Lemma 3.2]{BH18}. The maps \(\emptyset\to\ast\) and \(\ast\amalg\ast\to\ast\) visibly satisfy the conditions of \cref{defn:FixedN}. We therefore only have to show pullback stability. Let 
	\begin{equation}\label{eqn:PullBack}
	\begin{tikzcd}
		S_{1}
			\arrow{r}
			\arrow{d}[left]{g_1}
			&
		S
			\arrow{d}{g}
			\\
		T_{1}
			\arrow{r}
			&
		T
	\end{tikzcd}
	\end{equation}
	be a pullback diagram in \(\Set^{G}\) with \(g\) a map in \(\co^N_\gen\). Since fixed points are a limit, this gives a pullback diagram
	\begin{equation}\label{eqn:Fixedpb}
	\begin{tikzcd}
		{S_{1}^{N}}
			\arrow{r}
			\arrow{d}[left]{g_1}
			&
		{S^{N}}
			\arrow{d}{g}
			\\
		{T_{1}^{N}}
			\arrow{r}
			&
		{T^{N}}
	\end{tikzcd}
	\end{equation}
	Consider now the natural map
	\[
		S_{1}^{N}\to S_{1}\times_{T_{1}} T_{1}^{N}.
	\]
	By assumption on the diagram and associativity, we have natural isomorphisms
	\[
		S_{1}\times_{T_{1}} T_{1}^{N}\cong (S\times_{T}T_{1})\times_{T_{1}} T_{1}^{N}\cong S\times_{T} T_{1}^{N}\cong S\times_{T} (T^{N}\times_{T^{N}} T_{1}^{N}).
	\]
	By assumption on the map \(g\), we have a further natural isomorphism
	\[
		(S\times_{T}T^{N})\times_{T^{N}} T_{1}^{N}\cong S^{N}\times_{T^{N}} T_{1}^{N}.
	\]
	\eqref{eqn:Fixedpb} then shows this to be isomorphic to \(S_{1}^{N}\) via the natural maps, as desired.
\end{proof}

\subsection{Incomplete Tambara functors and nullifications}

In equivariant homotopy theory, the \(N\)-geometric fixed points are the composite of two functors:
\begin{enumerate}[(1)]
	\item the nullification which annihilates anything induced from those \(H\) in \(\mathcal F_N\);
	\item the ``restriction'' along the surjection \(q\colon G\to Q\).
\end{enumerate}
Both of these can be realized in Mackey functors, cf. \cite[Section 5.2]{BGHL19}.

\begin{definition}
	Let \(\underline{E\cf}_N\) be the sub-Mackey functor of \(\uA\) generated by \(\uA(G/H)\) for all \(H\) in \(\cf_N\). Let \(\underline{\widetilde{E}\cf}_N\) be the quotient of \(\uA\) by \(\underline{E\cf}_N\).
\end{definition}

Since we are working with the Burnside Mackey functor, we can be more explicit.

\begin{proposition}
	We have for all \(H\subset G\)
	\[
		\underline{E\cf}_N(G/H)=
		\mathbb Z\cdot\big\{
		[H/K] \mid K\in\cf_N
		\big\},
	\]
	so the inclusion
	\[
	\underline{E\cf}_N(G/H)\hookrightarrow\uA(G/H)
	\]
	is inclusion of a direct summand.
\end{proposition}
Note in particular that if \(H\not\in\cf_N\), then every generator of \(\underline{E\cf}_N(G/H)\) is the transfer of an element (in fact, a generator) from \(\underline{E\cf}_N(G/K)\) with \(K\in\cf_N\).

\begin{definition}[{\cite[Definition 5.6]{BH18}}]
	If $\uR$ is an $\co$-Tambara functor, then an \emph{$\co$-ideal} is a sub-Mackey functor $\underline{J}$ such that:
	\begin{enumerate}[(a)]
		\item the multiplication on $\uR$ makes $\underline{J}$ an $\uR$-module;
		\item if $f : S \to T$ is in $\co$ and is surjective, then $\underline{J}$ is closed under $\nm_f$, where $\nm_f$ is the norm along $f$ in $\uR$ (see \cref{Def:NormTransferRestriction}). 
	\end{enumerate}
\end{definition}

The surjective condition allows for ideals that don't contain units $1 \in \uR(T)$ for each $T$; see \cite[Remark 5.7]{BH18}.

\begin{theorem}
	For any normal subgroup \(N\) of \(G\), the sub-Mackey functor \(\underline{E\cf}_N\) of \(\uA\) is an \(\co^N_\gen\)-ideal.
\end{theorem}
\begin{proof}
	That this is a Mackey ideal follows from the Frobenius relation: the transfers from a family form an ideal by the double coset formula. 
	
	It remains to show that this is closed under norm maps parameterized by \(\co^N_\gen\): if \(x\in \underline{E\cf}_N(G/K)\), then any norm parameterized by \(\co^N_\gen\) applied to \(x\) is again in \(\underline{E\cf}_N\). For this, we first recall that an element \(x\in\uA(G/K)\) is actually in \(\underline{E\cf}_N(G/K)\) if and only if there is a finite \(G\)-set \(T\) with two properties:
	\begin{enumerate}[(1)]
	    \item the \(N\)-fixed points of \(T\) are empty;
	    \item there is a map \(h\colon T\to G/K\) such that \(x\) is in the image of the transfer along \(h\).
	\end{enumerate}
	
	If \(K\in\cf_N\), then the only norms from \(G/H\) parameterized by \(\co^N_\gen\) are those along maps \(G/K\to G/H\) with \(H\) also in \(\cf_N\). The values of \(\underline{E\cf}_N\) and \(\uA\) agree at these groups, so there is nothing to check.
	
	Now if \(H\) contains \(N\), then admissibility says that \(K\) must as well by \cref{prop:ONInclusion}. We are therefore reduced to understanding the composite
	\[
	N_g\circ T_h,
	\]
	where \(g\colon G/K\to G/H\) is arbitrary and \(h\colon T\to G/K\) is the map above describing some element \(x\in\underline{E\cf}_N(G/K)\). We have an exponential diagram
	\[
	\begin{tikzcd}
		{G/K}
	    	\arrow{d}[left]{g}
	    	&
		{T}
			\arrow{l}[above]{h}
			&
		{(h')^\ast G/K}
	    	\arrow{l}[above]{f'}	
			\arrow{d}{g'}
	    	\\
		{G/K}
		    &
	    	&
		{G\times_H \Set^K(H,T'),}
			\arrow{ll}{h'}
	\end{tikzcd}
	\]
    where \(T'\) is the \(K\)-set \(h^{-1}(eK)\). Closure under the norm associated to \(g\) is then equivalent to 
    \[
	    \big(G\times_H\Set^K(H,T')\big)^N=\emptyset.
    \]
    Since \(H\) and \(K\) both contain \(N\), it suffices to check 
    \[
    	\big(\Set^K(H,T')\big)^N=\emptyset.
    \]
    Both $N$-fixed points and coinduction are right adjoints, and since the corresponding left adjoints commute, we can swap these: 
	\[
		\big(\Set^K(H,T)\big)^N\cong \Set^{K/N}\big(H/N,T^N\big).
	\]
	By assumption $T$, and hence \(T'\), has no $N$ fixed points. 
\end{proof}

Applying \cite[Proposition 5.11]{BH18}, we obtain:

\begin{corollary}
	If \(\co\) is any indexing category contained in \(\co^N_\gen\), then for any \(\co\)-Tambara functor \(\uR\), \(\underline{\widetilde{E}\cf}_N\boxtimes\uR\) is an \(\co\)-Tambara functor and the natural map
	\[
		\uR\to \underline{\widetilde{E}\cf}_N\boxtimes\uR
	\]
	is a map of \(\co\)-Tambara functors.
\end{corollary}

\subsection{Changing Groups}

Let $N$ be a normal subgroup of $G$ and $q \colon G \to Q = G/N$ the quotient homomorphism. The functor \(q^\ast\) on orbit categories \eqref{eq:qstaronorbitcategories} induces a fully-faithful functor on their coproduct completions
\[
	q^\ast\colon\Set^Q\to\Set^G.
\]
The essential image of \(q^\ast\) is the full subcategory of \(G\)-sets \(T\) for which \(T=T^N\).

The functor \(q^\ast\) actually gives us a map the other way on indexing subcategories. 

\begin{proposition}
	Given an indexing subcategory \(\co\) of $\Set^G$, the intersection with \(\Set^Q\) gives an indexing subcategory of \(\Set^Q\).
\end{proposition}

\begin{proof}
	Since \(\Set^Q\) is fully-faithfully embedded in \(\Set^G\) via \(q^\ast\) and since \(\co\) is wide and finite-coproduct complete, so is the intersection with \(\Set^Q\). Pullback stability follows from noting that we can compute pullback via the fully-faithful embedding, where this is immediate.
\end{proof}

\begin{definition}
	If \(\co\) is an indexing category for \(G\), then let \(q_\ast\co\) be the corresponding indexing category for \(Q\).
\end{definition}

We can extend this to maps on Burnside categories and categories of polynomials.  

\begin{proposition}\label{prop:EmbeddingOnPoly}
	The natural inclusion \(\Set^Q\to\Set^G\) extends to a faithful, but not full, product-preserving embedding
	\[
		q^\ast\colon\ca^Q\to\ca^G.
	\]
	If \(\co\) is an indexing category for \(Q\) and if \(\co'\) is any indexing category for \(G\) such that \(\co\subseteq q_\ast\co'\), then we have a faithful product preserving embedding
	\[
		q^\ast\colon\cp^Q_{\co}\to \cp^G_{\co'}.
	\]
	In both cases, we take a diagram in \(\Set^Q\) to itself, viewed as a diagram in \(\Set^G\) via \(q^\ast\).
\end{proposition}

\begin{proof}
	We begin by noting that \(q^\ast\) preserves pullback diagrams. This gives the first embedding immediately.

	For the second, we need slightly more: we need that \(q^\ast\) preserves not only pullback diagrams but also dependent products. This follows from computing the \(N\)-fixed points.
\end{proof}

\begin{remark}
	The failure of fullness arises from the image of \(q^\ast\) not being a sieve. For example, if \(Q=e\), then this is the usual embedding of \(\Set\) into \(\Set^G\). For any non-trivial \(G\), 
	\[
		A(G)=\ca^G(*,*)\neq \ca^e(*,*)=\z.
	\]
\end{remark}

Since the functor \(q^\ast\) is product preserving, precomposition with it gives a functor on Mackey functors and appropriate incomplete Tambara functors.

\begin{definition}
	If \(\co\) is an indexing category for \(G\), then let
	\[
		q_\ast\colon \co\mhyphen\Tamb_G\to q_\ast\co\mhyphen\Tamb_Q
	\]
	be the functor given by precomposition with \(q^\ast\), and similarly for Mackey functors.
\end{definition}

\begin{remark}
	Unpacking the functor \(q_\ast\), we see that it is really formalizing two procedures:
	\begin{enumerate}[(a)]
		\item forget \(\uR(G/H)\) for any \(H\) that does not contain \(N\); 
		\item forget any norms or transfers up from groups which do not contain \(N\).
	\end{enumerate}
	The heart of \cref{prop:EmbeddingOnPoly} is that this actually gives an incomplete Tambara functor on \(Q\).
\end{remark}

\begin{remark}
	Our functor \(q^\ast\) is the \(1\)-categorical shadow of an identically defined map of \(\infty\)-categories from Barwick's Burnside \(\infty\)-category for \(Q\) to that for \(G\) \cite{Bar17}. Precomposition with this \(\infty\)-categorical \(q^\ast\) gives a model for the \(N\)-fixed points as a functor \(\Sp^G\to\Sp^Q\).
\end{remark}

This is the final piece of our geometric fixed points.

\begin{definition}[cf. {\cite[Section 5.2]{BGHL19}}]\label{Def:GFP}
	Let \(\co\) be an indexing category for \(G\) such that \(\co\subseteq \co^N_\gen\). The \emph{Tambara geometric fixed points functor} 
	\[
		\tilde{\Phi}^N\colon\co\mhyphen\Tamb_G\to q_\ast\co\mhyphen\Tamb_Q
	\]
	is the composite
	\[
		q_\ast\circ \big(\underline{\widetilde{E}\cf}_N\boxtimes(\mhyphen)\big).
	\]

	The \emph{Mackey geometric fixed points functor} 
	\[
		\Phi^N\colon \Mack_G\to\Mack_Q
	\]
	is the composite
	\[
		q_\ast\circ \big(\underline{\widetilde{E}\cf}_N\boxtimes(\mhyphen)\big).
	\]
\end{definition}

\begin{remark}\label{rem:qDoesNothing}
	Note that for any Mackey functor \(\uM\), the Mackey functor \(\underline{\widetilde{E}\cf}_N\boxtimes\uM\) vanishes when evaluated on \(G/H\) with \(N\not\subset H\). In particular, we see that on the essential image of the localization functor given by boxing with \(\underline{\widetilde{E}\cf}_N\), \(q_\ast\) throws away no real information. 
\end{remark}

Note that the embedding
\[
	\cp^G_{\Iso} \hookrightarrow \cp^G_{\co}
\]
is compatible with the map \(q^\ast\): for any \(G\) indexing category \(\co\), we have a commutative diagram
\[
\begin{tikzcd}
	{\cp^Q_{\Iso}}
		\arrow{r}
		\arrow{d}[left]{q^\ast}
		&
	{\cp^Q_{q_\ast\co}}
		\arrow{d}{q^\ast}
		\\
	{\cp^G_{\Iso}}
		\arrow{r}
		&
	{\cp^G_{\co}.}	
\end{tikzcd}
\]
Precomposition with the inclusion \(\cp^G_{\Iso} \hookrightarrow \cp^G_{\co}\) gives the underlying Mackey functor of a Tambara functor (\cref{Example:SetGIso}). 

\begin{proposition}
	The underlying Mackey functor of the Tambara geometric fixed points is the Mackey geometric fixed points of the underlying Mackey functor, i.e.
	\[
	U\circ \tilde{\Phi}^N\cong \Phi^N\circ U.
	\]
\end{proposition}

\subsection{Geometric fixed points of frees and flats}

Let \(T\) be a finite \(G\)-set. We now compute the geometric fixed points of the free Mackey or \(\co\)-Tambara functor on \(T\).  The key feature here is that since \(N\) is a normal subgroup, \(T^N\) is actually a \(G\)-equivariant summand of \(T\), and hence we have a natural inclusion of \(G\)-sets \(T^N\to T\).

\begin{theorem}\label{Thm:GFPFree}
	The map
	\[
		\uA^Q\{x_{T^N}\} \to \Phi^N\big(\uA \{x_T\}\big)
	\]
	corresponding to the element
	\[
		[T\leftarrow T^N\to T^N]\in \big(\underline{\widetilde{E}\cf}_N\boxtimes\uA\{x_T\}\big)(T^N)
	\]
	is an isomorphism.
\end{theorem}

\begin{proof}
	We analyze more directly the quotient \(\underline{\widetilde{E}\cf}_N\boxtimes\uA\{x_T\}\). A generic element of \(\uA\{x_T\}(G/H)\) is a span of the form
	\[
		T\leftarrow S\to G/H.
	\]
	If \(N\not\subset H\), then this span was automatically set to zero by the quotient, so it suffices to study \(N\subset H\). Breaking \(S\) into orbits \(\coprod G/K_i\), we can write the element as the sum of elements
	\[
		T\leftarrow G/K_i\to G/H.
	\]
	If \(N\not\subset K_i\), then this element is in the image of the transfer from \(\uA\{x_T\}(G/K_i)\) with \(K_i\in\cf_N\), and hence is set equal to zero in the quotient. Thus the image under the quotient map is 
	\[
		T\leftarrow S^N\to G/H.
	\]
	In particular, we see that a basis for the free abelian group \(\underline{\widetilde{E}\cf}_N\boxtimes\uA\{x_T\}(G/H)\) is given by isomorphism classes of spans
	\[
		T\leftarrow G/K\to G/H
	\]
	with \(N\subset K\). Since \((G/K)^N=G/K\), the map \(G/K\to T\) actually factors through \(T^N\hookrightarrow T\).

	By naturality, the map $\uA^Q\{x_{T^N}\} \to \Phi^N \left( \uA\{x_T\} \right) $, when evaluated at \(G/H\) with \(N\subset H\), takes a diagram
	\[
		T^N\xleftarrow{f} S\xrightarrow{g} G/H=T_g\circ R_f\big(T^N\leftarrow T^N\to T^N)
	\]
	with \(S=S^N\) to the diagram
	\[
		T_g\circ R_f(T\leftarrow T^N\to T^N)=T\xleftarrow{f} S\xrightarrow{g} G/H.
	\]
	In particular, this takes a basis to a basis, and hence is an isomorphism.
\end{proof}

\begin{corollary}\label{Cor:GeometricFixedPtsPreserveProjective}
    The \(N\)-geometric fixed points functor preserves projective Mackey functors.
\end{corollary}
\begin{proof}
    Recall that any projective is a retract of a free. Since any functor preserves retract diagrams, and since geometric fixed points of free Mackey functors are free, the geometric fixed points of a projective Mackey functor are projective.
\end{proof}

The argument for incomplete Tambara functors is almost identical.

\begin{theorem}\label{Thm:GFPFreeIncomplete}
	Let \(\co\) be an indexing category for \(G\) for which \(\co\subseteq \co^N_\gen\). The map
	\[
		\uA^{q_\ast\co}[x_{T^N}]\to \tilde{\Phi}^N\uA^{\co}[x_T]
	\]
	corresponding to the element
	\[
		[T\leftarrow T^N\to T^N\to T^N]\in\big(\underline{\widetilde{E}\cf}_N\boxtimes\uA^{q_\ast\co}[x_{T}]\big)(T^N),
	\]
	is an isomorphism.
\end{theorem}

\begin{proof}
	The proof proceeds almost identically to the Mackey functor case. Using the same reductions as before, we find that a basis for the quotient \(\underline{\widetilde{E}\cf}_N\boxtimes \uA^{\co}[x_T]\) at \(G/H\) is given by isomorphism classes of polynomials of the form
	\[
		T\leftarrow A\xrightarrow{g} G/K\to G/H,
	\]
	with again \(N\subset K, H\) and now \(g\in\co\). Here our assumption on \(\co\) enters: we have no maps \(G/J\to G/K\) in \(\co^N_\gen\) if  \(J\) does not contain \(N\). In particular, we deduce that \(A=A^N\), and hence it is in the image of \(q^\ast\). The rest of the proof follows identically.
\end{proof}

More generally, the geometric fixed points functor preserves flat objects. We begin with an observation linking this to another well-studied functor: inflation.

\begin{definition}[{\cite[Section 5]{TW90}}]
    Given a $Q$-Mackey functor $\uM$, we define a $G$-Mackey functor $\Inf_Q^G \uM$ called the \emph{inflation of $\uM$ from $Q$ to $G$} by
\[
	\Inf_Q^G \uM(G/H) = \begin{cases}
		0 \quad & \text{ if } H\in\mathcal F_N, \\
		\uM(H/N) \quad & \text{ if } H\not\in\mathcal F_N.
	\end{cases}
\]
\end{definition}

Th{\'e}vanaz--Webb show also that this functor has both adjoints. Unpacking our definition of geometric fixed points shows that it agrees with the left adjoint of inflation.

\begin{proposition}
    The \(N\)-geometric fixed points on Mackey functors is left adjoint to the inflation functor \(\Inf_Q^G\).
\end{proposition}

In fact, we can do better. Inflation actually gives a section of the \(N\)-geometric fixed points (which reflects an underlying recollement for the family \(\mathcal F_N\)).

\begin{proposition}\label{prop:InfSplitsPhi}
    The canonical natural transformation
    \[
    \Inf_Q^G(\mhyphen)\cong \uA\boxtimes \Inf_Q^G(\mhyphen)\Rightarrow \ul{\tilde{E}\mathcal F}_N\boxtimes \Inf_Q^G(\mhyphen)
    \]
    is an isomorphism, and the composite \(\Phi^N\circ \Inf_Q^G\) is naturally isomorphic to the identity.
\end{proposition}

\begin{proof}
    For any \(G\)-Mackey functor \(\uM\), the canonical quotient map \(\uA\to \ul{\tilde{E}\mathcal F}_N\) allows us to identify 
    \[
    \ul{\tilde{E}\mathcal F}_N\boxtimes \uM
    \]
    with the quotient of \(\uM\) by the sub-Mackey functor generated by \(\uM(G/H)\) for all \(H\in\mathcal F_N\). By definition, if \(H\in\mathcal F_N\), then \(\Inf_Q^G\uM(G/H)=0\), so we are forming the quotient by the zero Mackey functor.
\end{proof}

Identifying the geometric fixed points as the left-adjoint to inflation gives another reformulation of it. 

\begin{definition}
    Let \(\FixN\colon\Set^G\to\Set^Q\) denote the \(N\)-fixed point functor.
\end{definition}

\begin{proposition}\label{Prop:FixNExtends}
    The functor \(\FixN\) extends to a product-preserving functor
    \[
        \mathcal A^G\to\mathcal A^Q
    \]
    which commutes with Cartesian products.
\end{proposition}

\begin{proof}
    The composition in \(\mathcal A^G\) is given by pullback, and since \(\FixN\) is a limit, it commutes with pullbacks. The compatibility with the Cartesian product is identical. For the categorical product, we observe that the fixed points of a disjoint union are the disjoint union of the fixed points.
\end{proof}

The following is immediate from the definition of inflation.
\begin{proposition}
    The functor \(\Inf_Q^G\) is naturally isomorphic to the precomposition with \(\FixN\):
    \[
        \Inf_Q^G(\uM)\cong \uM\circ \FixN.
    \]
\end{proposition}

\begin{corollary}
    The \(N\)-geometric fixed points are given by the left Kan extension along \(\FixN\).
\end{corollary}

\begin{corollary}\label{cor:GFPStrongSymMonoidal}
    The \(N\)-geometric fixed points functor is strong symmetric monoidal.
\end{corollary}
\begin{proof}
    Since both the box product and the \(N\)-geometric fixed points are given by left Kan extensions, it suffices to show that the underlying diagram 
    \[
    \begin{tikzcd}
        {\ca^G\times\ca^G}
            \arrow{r}{\times}
            \arrow{d}[left]{\FixN}
            &
        {\ca^G}
            \arrow{d}{\FixN}
            \\
        {\ca^Q\times\ca^Q}
            \arrow{r}[below]{\times}
            &
        {\ca^Q,}
    \end{tikzcd}
    \]
    expressing the ways we can take the iterated Kan extensions, commutes. This is the fact that \(\FixN\) is strong symmetric monoidal for the Cartesian product.
\end{proof}

\begin{theorem}\label{thm:GeomFPPreservesFlats}
    The \(N\)-geometric fixed points preserves flat Mackey functors.
\end{theorem}

\begin{proof}
    We must show that if \(\uM\) is a flat Mackey functor, then the box product with \(\Phi^N\uM\) is an exact functor on \(Q\)-Mackey functors. We show this by rewriting it several ways. Let \(\uN\) be a \(Q\)-Mackey functor. Then using \cref{prop:InfSplitsPhi} and \cref{cor:GFPStrongSymMonoidal}, we have a natural (in \(\uN\)) isomorphism
    \[
    \uN\boxtimes\Phi^N\uM\cong \Phi^N\big(\Inf_Q^G\uN\boxtimes \uM\big).
    \]
    The definition of geometric fixed points allows us to further rewrite this, now using the first clause of \cref{prop:InfSplitsPhi}:
    \[
    \Phi^N\big(\Inf_Q^G\uN\boxtimes \uM\big)\cong q_\ast \big(\Inf_Q^G\uN\boxtimes\uM\big).
    \]
    Now, the inflation functor and \(q_\ast\) are both exact, since both are given by precomposition with an additive functor and exactness is checked objectwise. We deduce that if \(\uM\) is flat, then the functor
    \[
        \uN\mapsto q_\ast\big(\Inf_Q^G\uN\boxtimes\uM\big)\cong \uN\boxtimes\Phi^N(\uM)
    \]
    is exact.
\end{proof}

\section{Necessary Conditions for Freeness}\label{Sec:Converse}
\subsection{The restriction functor from $G$-Mackey functors to $H$-Mackey functors}

To give necessary conditions for flatness, we must first establish some properties of the restriction functor. Recall that $i_H^* \colon \Mack_G \to \Mack_H$ for $H \leq G$ is defined by precomposition with the induction functor 
\[ G \times_H (-) \colon \ca^H \to \ca^G. \]
We first introduce its adjoints. 

\begin{definition}[{\cite[Section 4]{TW90}}]
	Given an $H$-Mackey functor $\uN$, we define a $G$-Mackey functor $\Ind_H^G(\uN)$, the \emph{induction} of $\uN$, by 
	\[ 
		\Ind_H^G\uN  = \uN \circ i_H^\ast,
	\]
	where $i_H^\ast$ is the restriction functor from finite $G$-sets to finite $H$-sets.
\end{definition}

\begin{lemma}[{\cite[Proposition 4.2]{TW90}}]
	The restriction functor is both left and right adjoint to the induction functor.
\end{lemma}

Both restriction and induction are exact functors, as again, exactness is checked objectwise. 

\begin{proposition}
    For any subgroup \(H\), both restriction \(i_H^\ast\) and induction \(\Ind_H^G\) preserve projective objects.
\end{proposition}
\begin{proof}
    This follows from each having an exact right adjoint.
\end{proof}

We can moreover identify the restriction of free Mackey functors.

\begin{proposition}
	There is an isomorphism of Mackey functors 
	\[ 
		i_H^*\uA^G\{x_T\} \cong \uA^H\{x_{i^*_HT}\}.
	\]
	Moreover, $i_H^*$ sends free Mackey functors to free Mackey functors.
\end{proposition}

\begin{proof}
	It suffices to show that $i_H^*\uA^G\{x_T\}$ represents evaluation at $i_H^*T$ in $H$-Mackey functors. We compute:
	\begin{align*}
		\Mack^H(i_H^*\uA^G\{x_T\},\uN) & 
			\cong \Mack^G(\uA^G\{x_T\}, \Ind_H^G\uN) \\
			\cong \Ind_H^G\uN(T) \\
			\cong \uN(i_H^\ast T).
	\end{align*}
	If a free Mackey functor is finitely generated, the calculation above shows that its restriction is again free. Otherwise, it is a direct limit of finitely generated ones, and $i_H^*$ preserves direct limits because it is a left adjoint.
\end{proof}

As with geometric fixed points, it is a bit more work to show that this functor preserves flat objects.

\begin{lemma}
	The restriction functor $i_H^*$ is strong symmetric monoidal.
\end{lemma}

\begin{proof}
	Because restriction is left adjoint to a functor $\Ind_H^G$ given by precomposition with $i_H^\ast$, it is given by a left Kan extension along $i_H^\ast \colon \Set^G \to \Set^H$. Then this follows nearly identically to \cref{cor:GFPStrongSymMonoidal}, replacing instances of $\FixN$ by $i_H^\ast$. Note that $i_H^\ast \colon \Set^G \to \Set^H$ is a right adjoint and so preserves limits, and so extends to a product-preserving functor $\ca^G \to \ca^H$ which commutes with Cartesian products as in \cref{Prop:FixNExtends}.
\end{proof}

\begin{theorem}\label{Thm:RestrictionPreservesFlats}
	If $\uM$ is a flat $G$-Mackey functor, then $i_H^*\uM$ is a flat $H$-Mackey functor. 
\end{theorem}

\begin{proof}
	Since induction and restriction are exact functors, if \(\uM\) is a flat \(G\)-Mackey functor, then the functor
	\[
		\uN \mapsto i_H^\ast\big(\Ind_H^G(\uN)\boxtimes\uM\big)
	\]
	is exact. By the Frobenius relation, we have an isomorphism
	\[
	    \Ind_H^G(\uN)\boxtimes\uM\cong \Ind_H^G\big(\uN\boxtimes i_H^\ast \uM\big).
	\]
	Finally, the unit of the restriction-induction adjunction
	\[
	\uN\mapsto i_H^\ast\Ind_H^G \uN
	\]
	is naturally a split inclusion. This follows from noting the that right-hand side is the functor
	\[
	    T\mapsto \uN\big(i_H^\ast(G\times_{H} T)\big),
	\]
	and the inclusion of \(H\) into \(G\) is the inclusion of an \(H\mhyphen H\)-biset summand. We deduce that the functor
	\[
		\uN\mapsto \uN \boxtimes i_H^\ast\uM
	\]
	is a retract of an exact functor, and hence is exact.
\end{proof}

\subsection{Necessity of triviality of the restriction}

We begin by proving that $i^*_H \co \cong \co^{\triv}$ is a necessary condition for freeness. 
We start here with a technical observation:

\begin{lemma}\label{lem:SmallNormal}
    If \(\co\) is an indexing category for \(G\), then there is a smallest subgroup \(N\) such that
    \begin{enumerate}[(a)]
        \item \(G/N\) is admissible.
        \item If \(G/H\) is admissible, then \(N\leq H\).
        \item If \(H/K\) is an admissible \(H\)-set with \(N\leq H\), then \(N\leq K\).
        \item $N$ is normal in $G$.
    \end{enumerate}
\end{lemma}

\begin{proof}
	Let $S = \{ H \leq G \mid G/H \text{ is admissible} \}$. Let $N = \bigcap_{H \in S} H$. Then $N$ satisfies $(a)$ and $(b)$:
	\begin{enumerate}[(a)]
		\item $G/N$ is admissible since $G/(H \cap H')$ is admissible whenever $G/H$ and $G/H'$ are admissible.
		\item By definition, $N \leq H$ for every $H$ such that $G/H$ is admissible.
	\end{enumerate}
	Now, suppose $H/K$ is an admissible $H$-set and $N \leq H$. By restriction, \(N/(N\cap K)\) is an admissible \(N\)-set, and since admissibles are closed under self-induction, \(G/(N\cap K)\) is an admissible \(G\)-set. Minimality of \(N\) implies that \(N\cap K=N\), and hence \(N\subset K\).

	Finally, the closure under conjugacy of admissibles implies by minimality that \(N\) is actually normal.
\end{proof}

As in \cref{Sec:GFP}, we write \(Q\) for \(G/N\).

\begin{remark}
	If \(N\) is the smallest normal subgroup associated to \(\co\) from \cref{lem:SmallNormal}, then
	\[
		\co\cap \co^N_\gen=\co.
	\]
	In particular, we can take  geometric fixed points and not lose information we might have wanted to preserve.
\end{remark}

\begin{proposition}\label{Prop:Triv}
    Let \(\co\) be an indexing category. If \(\uA^{\co}[x_{G/G}]\) is flat, then \(\co=\co^{\triv}\).
\end{proposition}

\begin{proof}
    Let \(N\) be the (normal) subgroup from \cref{lem:SmallNormal}, and take the \(N\)-geometric fixed points. This gives
    \[
	    \Phi^N\big(\uA^{\co}[x_{G/G}]\big)\cong \uA^{\co^N}[x_{Q/Q}],
    \]
    which is flat by \cref{thm:GeomFPPreservesFlats}. Note that now we have norms from \(N/N=\{e\}\) to all larger subgroups of \(Q\). 
    
    If \(N\neq G\), then we can choose a \(C_p\) in \(Q\) for some prime \(p\). Restricting down to \(C_p\), we get
    \[
    	i_{C_p}^\ast\uA^{q_*\co}[x_{Q/Q}]\cong \uA^{\co^{\cplt}}[x_{C_p/C_p}],
    \]
    since there are only two indexing categories for \(C_p\) and we have a non-trivial norm. This is not flat by \cite[Lemma 3.6]{BH18}, which implies that the augmentation ideal of $\uA \to \uZ$ (which is not flat) is a summand. We conclude that \(N=G\) and hence $G/G$ is the only admissible transitive $G$-set for $\co$.

    To show that the admissible \(H\)-sets are also all trivial, we note that we have a natural isomorphism
    \[
        i_H^\ast\uA^{\co}[x_{G/G}]\cong \uA^{i_H^\ast\co}[x_{H/H}].
    \]
    By \cref{Thm:RestrictionPreservesFlats} restriction preserves flats, so induction on the subgroup lattice shows then that \(\co=\co^{\triv}\).
\end{proof}

\begin{corollary}\label{Cor:TrivRes}
    Let \(\co\) be an indexing category and let \(H\leq G\) be a subgroup. If \(\uA^{\co}[x_{G/H}]\) is flat, then \(i_H^\ast\co=\co^{\triv}\).
\end{corollary}

\begin{proof}
    Write \(i_H^\ast G/H=H/H\amalg T\) for some \(H\)-set \(T\). If \(\uA^{\co}[x_{G/H}]\) is flat, then so is the restriction to \(H\). This is given by
    \[
        \uA^{i_H^\ast\co}[x_{i_H^\ast G/H}]\cong \uA^{i_H^\ast\co}[x_{H/H}]\boxtimes \uA^{i_H^\ast\co}[x_{T}].
    \]
    If this is flat, then so are both box factors since they are also summands. The result then follows from \cref{Prop:Triv}.
\end{proof}

\subsection{Necessity that $H$ is normal in $G$}

In fact, the decomposition in the proof of \cref{Cor:TrivRes} implies that $H$ is a normal subgroup of $G$. Consequently, all of the factors in that decomposition are the same (and the Weyl action permutes them). To show this, we need to analyze exactly when Green functors are free.

\begin{theorem}\label{thm:NoGo}
	If \(\uA^{\co^\triv}[x_{G/H}]\) is flat, then $H = G$.
\end{theorem}

The free Green functor on a class at level \(G/H\) can be equivalently described as the ordinary symmetric algebra on \(\uA\{x_{G/H}\}\):
\[
	\uA^{\co^\triv}[x_{G/H}]\cong \bigoplus_{n\geq 0} \Sym^n(\uA\{x_{G/H}\}).
\]

For a general finite \(G\)-set \(T\), the \(n\)-th symmetric power on \(\uA\{x_T\}\) is a quotient of a free:
\[
	\Sym^n(\uA\{x_T\})\cong \big(\uA\{x_{T^{\times n}}\}\big)/\Sigma_n.
\]
Here we use that \(T^{\times n}\) is a \(G\times\Sigma_n\)-set, or equivalently, a \(\Sigma_n\)-object in \(G\)-sets.

\cref{thm:NoGo} will follow immediately from a more precise statement.

\begin{theorem}
	If \(n=[G:H]\), then \(\Sym^n(\uA\{x_{G/H}\})\) is not flat.
\end{theorem}

\begin{proof}
	A decomposition of \((G/H)^{\times n}\) into \(G\times\Sigma_n\)-sets gives a direct sum decomposition of \(\Sym^n(\uA\{x_{G/H}\})\). We single out some particular summands.

	Choose a non-equivariant isomorphism
	\[
		\big(f\colon \{1,\dots,n\}\to G/H\big)\in (G/H)^{\times n}.
	\]
	Using \(f\) to identify \(G/H\) with \(\{1,\dots,n\}\), the \(G\)-action on \(G/H\) is classified by a homomorphism
	\[
		\phi\colon G\to\Sigma_n.
	\]
	Thinking of \(\phi\) as defining the \(G\)-set structure on \(G/H\), we see that the kernel of \(\phi\) is the normal subgroup
	\[
		N_H=\bigcap_{g\in G} gHg^{-1}.
	\]
	Let \(\tilde{\phi}\) be the inclusion of \(G/N_H\) into \(\Sigma_n\) induced by \(\phi\).

	We have to understand the summand
	\[
		\big(\uA\{x_{G\times\Sigma_n\cdot f}\}\big)/\Sigma_n.
	\]
	of \(\Sym^n(\uA\{x_{G/H}\})\), where $G \times \Sigma_n \cdot f$ is the orbit of $f \in (G/H)^{\times n}$ under the $G \times \Sigma_n$-action. 
	
	By the orbit-stabilizer theorem, as a \(G\times\Sigma_n\)-set, we have
	\[
		G\times\Sigma_n\cdot f\cong G\times\Sigma_n/\Gamma_{\phi},
	\]
	where \(\Gamma_\phi\) is the graph of \(\phi\). The underlying Mackey functor is the free Mackey functor  
	\[
		\uA\{x_{i_G^\ast(G\times\Sigma_n\cdot f)}\},
	\]
	and this has a residual \(\Sigma_n\)-action. 

	Restricting to \(G\), we use the double-coset formula
	\[
		i_G^{\ast} \big(G\times\Sigma_n/\Gamma_{\phi}\big)\cong \coprod_{G(g,\sigma)\Gamma_\phi\in G\backslash G\times\Sigma_n/\Gamma_\phi} G/G\cap (g,\sigma)\Gamma_\phi (g,\sigma)^{-1}.
	\]
	The group \(G\cong G\times\{e\}\) is a normal subgroup of \(G\times\Sigma_n\), so it intersects all conjugates of a fixed group in conjugates:
	\[
		G\times\{e\}\cap ((g,\sigma)\Gamma_\phi (g,\sigma)^{-1})=(g,\sigma)\big(G\times\{e\}\cap\Gamma_\phi\big){(g,\sigma)^{-1}}.
	\]
	By definition of the graph, the intersection 
	\[
		G\times\{e\}\cap \Gamma_\phi=\{(g,e)\mid \phi(g)=e\}
	\]
	is the graph of \(\phi\) restricted to its kernel. Since the kernel is normal in \(G\), this is normal in \(G\times\Sigma_n\):
	\[
		\big(G\times\{e\}\cap\Gamma_\phi\big)^{(g,\sigma)}=N_H\times\{e\}.
	\]
	Therefore we have
	\[
		i_G^{\ast} ( G\times\Sigma_n/\Gamma_{\phi} )\cong \coprod_{G(g,\sigma)\Gamma_\phi\in G\backslash G\times\Sigma_n/\Gamma_\phi} G/N_H.
	\]

	For the residual \(\Sigma_n\)-action, we recall that \(\Sigma_n\) acted freely on \(G\times\Sigma_n/\Gamma_\phi\). The quotient map 
	\[
		G\times\Sigma_n\to G\times\Sigma_n/\Gamma_\phi
	\]
	takes \(G\) first to \(G/N_H\) and then identifies it with the subgroup \(\tilde{\phi}(G/N_H)\). This identifies the double cosets: we have
	\[
		G\backslash G\times\Sigma_n/\Gamma_\phi=\Sigma_n/\tilde{\phi}(G/N_H).
	\]

	Putting this together, the restriction to \(G\) of \(G\times\Sigma_n/\Gamma_{\phi}\) is just the decomposition of \(\Sigma_n\) into \(G/N_H\)-cosets. As a consequence, we have
	\[
		\big(\uA\{x_{G\times\Sigma_n\cdot f}\}\big)/\Sigma_n\cong \big(\uA\{x_{G/N_H}\}\big)/(G/N_H),
	\]
	where \(G/N_H\) acts on itself via the right action (which is then an isomorphism of left \(G\)-sets). 

	If this summand is flat, then applying \(N_H\)-geometric fixed points (and letting \(Q=G/N_H\)) shows that the \(Q\)-Mackey functor
	\[
		\big(\uA^{Q}\{x_{Q}\}\big)/Q
	\]
	is also flat. In other words, without loss of generality, we reduce to the case that \(N_H=\{e\}\) and \(Q=G\).

	Since \(G\) is finite and \(G\neq \{e\}\), we can find an element of order \(p\) for some prime \(p\). This gives us a subgroup \(K\cong C_p\). Restricting to \(K\), we have an isomorphism
	\[
		i_K^\ast \big(\uA\{x_{G}\}\big)/G\cong \big(\uA\{x_K\}\big)/K.
	\] 
	The latter we can compute directly: it is the dual to the constant Mackey functor \(\uZ\). Since this is not flat, by \cref{Lem:ZDualNotFlat}, we deduce that the summand we started with could not be either. 
\end{proof}

This has a surprising consequence.

\begin{proposition}\label{Prop:HNormal}
    Let \(\co\) be an indexing category, and let \(H\subset G\).
    If \(\uA^{\co}[x_{G/H}]\) is flat, then \(H\) is normal in \(G\).
\end{proposition}

\begin{proof}
    Let us again write \(i_H^\ast G/H=H/H\amalg T\) for some finite \(H\)-set \(T\). We have an isomorphism of \(i_H^\ast\co\)-Tambara functors
    \[
	    i_H^\ast \uA^{\co}[x_{G/H}]\cong \uA^{i_H^\ast\co}[x_{H/H}]\boxtimes \uA^{i_H^\ast\co}[x_T],
    \]
    as in \cref{Cor:TrivRes}, and that corollary also shows that \(i_H^\ast\co=\co^\triv\). Since all of the box factors are also direct summands, \cref{thm:NoGo} shows that they are only free if the set \(T\) is a trivial \(H\)-set. This is equivalent to \(H\) being normal in \(G\).
\end{proof}

\subsection{Connecting freeness to admissibility of \(G/H\)}

We begin with a further constraint on the indexing category.

\begin{proposition}
     Let \(\co\) be an indexing category and \(N\) a normal subgroup of \(G\) such that \(i_N^\ast\co\) is trivial. Then 
     \[
	     \co\cap\co^N_\gen=\co.
     \]
\end{proposition}

\begin{proof}
	If a map of orbits $G/K \to G/H$ is in $\co$, then $K \not \in N$ since $i_N^*\co$ is trivial. To show that $G/K \to G/H$ is in $\co^N_\gen$, it suffices to show that $N \not \subseteq H$ by \cref{prop:ONInclusion}(b)(ii). If $N$ were contained in $H$, then pullback of $G/H \to G/K$ along $G/N \to G/H$ yields a map $G/L \to G/N$ in $\co$ with $L \neq N$, which does not exist in $\co$ because $i_N^*\co$ is trivial. 
\end{proof}

If \(\uA^{\co}[x_{G/H}]\) is flat, then we have already seen that \(H\) is normal in \(G\) and that \(i_H^\ast\co=\co^\triv\). We want to study the constraints on \(\co\), so we lose no information if we pass to the \(H\)-geometric fixed points. Put another way, we may assume without loss of generality that \(H=\{e\}\).

To study the connection between flatness and admissibility, we will compare $\co$ with an indexing category where we know that we have flatness: the complete one. We need a small lemma on how the free incomplete Tambara functors compare.

\begin{lemma}
	If $\co \subseteq \co'$, then for all $T$, we have a natural inclusion 
	\[
		\uA^{\co}[x_T] \hookrightarrow \uA^{\co'}[x_T]
	\]
	adjoint to 
	\[
		[T \leftarrow T \rightarrow T \rightarrow T] \in \uA^{\co'}[x_T](T).
	\]
	In particular, for any indexing category $\co$, we have an inclusion
	\[
		\phi_{\co}: \uA^{\co}[x_{G/e}] \hookrightarrow \uA^{\co^{\cplt}}[x_{G/e}] \cong N_e^G \z[x]
	\]
	into the free Tambara functor on an underlying generator. 
\end{lemma}

We will use this map to identify summands of \(\uA^{\co}[x_G/e]\) by studying their image. If \(G\) is admissible, then we are of course done. So, we describe what happens when \(G/e\) is not admissible.

\begin{definition}
	Let
	\[
		\cf_{\co} = \{ H \mid G/e \xrightarrow{\pi} G/H \in \co \},
	\]
	be the collection of subgroups \(H\) for which \(H\) is an admissible \(H\)-set.
\end{definition}

Put another way, these are all of the subgroups for which we have norms from the trivial group to those subgroups.

\begin{lemma}
	The collection of subgroups $\cf_{\co}$ of $G$ is a family.
\end{lemma}

\begin{proof}
	The result follows from pullback stability of $\co$. 
\end{proof}

\begin{proposition}
	Consider the summand 
	\[
		\uA\cong \uA \{\nm_e^G(x)\} \hookrightarrow N_e^G\big(\z[x]\big)
	\]
	corresponding to the norm of the (underlying) generator \(x\). 

	The image of $\phi_{\co}$ in $\uA\{\nm_e^G(x)\}$ is isomorphic to $\underline{E\cf}_{\co}$. 
\end{proposition}

\begin{proof}
	The only polynomials in \(\uA^{\co}[x_{G/e}]\) which land in this summand are the linear combinations of polynomials of the form
	\[
		[G/e \leftarrow G/e \times G/H \xrightarrow{\pi} G/H \rightarrow G/H], \quad H \in \cf_{\co}.
	\]
	In \(\uA \cdot \nm_e^G(x) \), the displayed polynomial is the restriction to \(H\) of \(\nm_e^G(x)\), so we see that the image is the sub-Mackey functor generated by \(\uA(G/H)\) with \(H \in \cf_{\co}\). The result follows. 
\end{proof}

\begin{proposition}
	If $G$ is solvable and \(G\not\in\cf_{\co}\), then $\underline{E\cf}_{\co}$ is not flat. 
\end{proposition}

\begin{proof}
	Assume to the contrary that $\underline{E\cf}_{\co}$ is flat. 

	Let \(H\) be a minimal subgroup of \(G\) that is not in \(\cf_{\co}\). By assumption, all proper subgroups of \(H\) are in \(\cf_{\co}\), and since restriction preserves flat objects (\cref{Thm:RestrictionPreservesFlats}), we may without loss of generality assume that \(\cf_{\co}\) is actually the family $\cp$ of proper subgroups of \(G\).

	Since $G$ is solvable, there exists some normal subgroup $N$ of $H$ such that $H/N \cong C_p$ for some prime \(p\). Since geometric fixed points $\Phi^N$ preserve flatness (\cref{thm:GeomFPPreservesFlats}), $\Phi^N(\underline{E\cp})$ is flat. However, direct computation shows that 
	\[
		\Phi^N(\underline{E\cp}) \cong \underline{\z}^\ast,
	\]
	which is not flat by \cref{Lem:ZDualNotFlat}. \end{proof}

We assemble all of the results here into the solvable case of our general theorem.

\begin{theorem}\label{Thm:SolvableCase}
Let \(G\) be a solvable finite group, $\co$ an indexing category for $G$, and $H$ a subgroup of $G$. The following are equivalent: 
\begin{enumerate}[(a)]
    \item The Mackey functor underlying the free \(\co\)-Tambara functor on a class at level \(G/H\) is flat.
    \item The Mackey functor underlying the free \(\co\)-Tambara functor on a class at level \(G/H\) is free.
    \item \(H\) is a normal subgroup of \(G\), \(G/H\) is admissible, and \(i_H^\ast\co=\co^{\triv}\).
\end{enumerate} 
\end{theorem}

\begin{remark}
	For more general groups, we notice that if \(G\) is not admissible, then we have a summand of the free \(\co\)-Tambara algebra on a class at level \(G\) that is a proper sub-Mackey functor of \(\uA\).

	In the non-solvable case, the idempotent splitting of the Burnside ring due to Dress \cite{Dre69} shows that the Burnside Mackey functor splits into a product of various Green functors, with factors corresponding to perfect subgroups. This means that there are, in the non-solvable case, sub-Mackey functors of \(\uA\) that are projective.

	For example, let \(G=A_5\) and let \(\co\) be the indexing category \(\co^G_\gen\). Then the free \(\co\)-Tambara functor at level \(G\) is underlying projective. The argument above shows that the part in \(\uA \{ \nm_{e}^{G}(x) \} \) is \(\ul{E\mathcal P}\), which for \(A_5\) is actually a direct summand.
\end{remark}

\subsection{Asymptotics}\label{SS:Asymptotics}

The conditions on \(H\) in \cref{Thm:SolvableCase} pin down \(H\) quite nicely: it must be the minimal normal subgroup \(N \leq G\) such that \(G/N\) is admissible. In particular, it is uniquely determined by the indexing category. 

\begin{proposition}\label{prop:OnlyOneSubgroup}
    If \(G\) is a solvable group, then for any indexing category \(\co\), there is at most one subgroup \(H \leq G\) such that \(\uA^{\co}[x_{G/H}]\) is flat as an \(\uA\)-module.
\end{proposition}

If the restriction of \(\co\) to the minimal normal subgroup \(N \leq G\) such that \(G/N\) is admissible is not trivial, then \cref{Thm:SolvableCase} shows that there are no free \(\co\)-Tambara functors that are flat as \(\uA\)-modules.

These harsh conditions let us deduce harsh upper bounds on the number of frees that are underlying flat in general. Our strategy is to bound the proportion of free incomplete Tambara functors which are flat as Mackey functors in terms of the depth of a finite solvable group $G$. We then use a result from group theory to reduce from general finite groups to finite solvable groups. 

\begin{definition}
	Let $G$ be a finite group. Let:
	\begin{enumerate}[(a)]
		\item $d(G)$ denote the depth of the subgroup lattice of $G$;
		\item $n(G)$ denote the number of indexing categories for $G$;
		\item $T(G)$ denote the the number of pairs $(\co, H)$ of indexing category $\co$ for $G$ and subgroup $H \leq G$;
		\item $P(G)$ denote the number of such pairs for which $\uA^{\co}[x_{G/H}]$ is flat as a Mackey functor. 
	\end{enumerate}
\end{definition}

The number $P(G)/T(G)$ is the fraction of free incomplete Tambara functors for $G$ which are flat as Mackey functors. The following proposition says that as the depth of a solvable group $G$ increases, this fraction tends to zero. 

\begin{proposition}\label{Prop:DepthBoundG}
	Let $G$ be a solvable group. Then
	\[
		\dfrac{P(G)}{T(G)} \leq \dfrac{1}{d(G)}.
	\]
\end{proposition}

\begin{proof}
	For any finite group $G$, we have $T(G) \geq d(G) \cdot n(G)$. It therefore suffices to show that when $G$ is solvable, we have $P(G) \leq n(G)$. This follows from \cref{prop:OnlyOneSubgroup}.
\end{proof}

\begin{definition}
	Let:
	\begin{enumerate}[(a)]
		\item $\g$ (resp. $\g^s$) denote the (countable) set of all finite groups (resp. finite solvable groups);
		\item $\g_{\leq d}$ (resp. $\g_d$) the subset of finite groups of depth at most $d$ (resp. precisely $d$);
		\item $\g_{\leq d}^s$ (resp. $\g_d^s$) the subset of finite solvable groups of depth at most $d$ (resp. precisely $d$). 
	\end{enumerate}
	For any of the sets $S$ above, we write $S(-): \n \to S$ for a fixed bijection with the natural numbers.
\end{definition}

Heuristically, the inequality $P(G)/T(G) \leq 1/d(G)$ says that as depth increases, the proportion of free incomplete Tambara functors which are flat decreases. To extend this heuristic to a precise statement for all finite groups, we will need to take sums over infinite sets of finite groups. This is possible using the following lemma:

\begin{lemma}\label{Lem:Sums}
	Let $S$ be a countable set and let $\{A_s\}_{s \in S}$ and $\{B_s\}_{s \in S}$ be collections of finite integers. Fix a bijection $\phi: \n \to S$ between $S$ and the natural numbers. Then
	\[
		\lim_{n \to \infty} \dfrac{ \sum_{i=1}^n A_{\phi(i)}}{ \sum_{i=1}^n B_{\phi(i)}} \leq \lim_{n \to \infty} \max_n \dfrac{A_{\phi(n)}}{B_{\phi(n)}}.
	\]
\end{lemma}

\begin{proof}
	By taking limits, it suffices to show that 
	\[ 
		\frac{\sum_{i=1}^n A_{\phi(i)}}{\sum_{i=1}^n B_{\phi(i)}} \leq \max_{1 \leq i \leq n} \frac{A_{\phi(i)}}{B_{\phi(i)}} 
	\]
	for all $n$. To that end, observe 
	\[
		A_{\phi(j)} = \frac{A_{\phi(j)}}{B_{\phi(j)}} B_{\phi(j)}  \leq  \left(\max_{1 \leq i \leq n} \frac{A_{\phi(i)}}{B_{\phi(i)}} \right) B_{\phi(j)}. 
	\]
	Taking sums, we find 
	\[ 
		\sum_{j=1}^n A_{\phi(j)} \leq \left(\max_{1 \leq i \leq n} \frac{A_{\phi(i)}}{B_{\phi(i)}} \right) \left(\sum_{j=1}^n B_{\phi(j)} \right). 
	\]
	Dividing through by $\sum_j B_{\phi(j)}$ yields the desired inequality. 
\end{proof}

\begin{theorem}\label{Thm:Asymptotics}
	Free incomplete Tambara functors for finite groups are almost never flat. More precisely, 
	\[
		\lim_{n \to \infty} \dfrac{\sum_{i=1}^n P(\g(i))}{\sum_{i=1}^n T(\g(i))} = 0.
	\]	
\end{theorem}

\begin{proof}
	Evidently $P(G)$ and $T(G)$ are positive numbers for all $G \in \g$, so 
	\[
		\lim_{n \to \infty} \dfrac{\sum_{i=1}^n P(\g(i))}{\sum_{i=1}^n T(\g(i))} \geq 0.
	\]
	We need to show
	\[
		\lim_{n \to \infty} \dfrac{\sum_{i=1}^n P(\g(i))}{\sum_{i=1}^n T(\g(i))} \leq 0.
	\]

	By the main theorem of \cite{CEG86}, almost all finite groups are solvable. Therefore 
	\[
		\lim_{n \to \infty} \dfrac{\sum_{i=1}^n P(\g(i))}{\sum_{i=1}^n T(\g(i))} = \lim_{n \to \infty} \dfrac{\sum_{i=1}^n P(\g^s(i))}{\sum_{i=1}^n T(\g^s(i))},
	\]
	so it suffices to prove the theorem for solvable groups. 

	Now, $\g^s = \bigcup_{d =0}^\infty \g^s_d$ so
	\[
		\lim_{n \to \infty} \dfrac{\sum_{i=1}^n P(\g^s(i))}{\sum_{i=1}^n T(\g^s(i))} = \lim_{d \to \infty} \lim_{n \to \infty} \dfrac{\sum_{i=1}^n P(\g^s_{\leq d}(i))}{\sum_{i=1}^n T(\g^s_{\leq d}(i))}.
	\]

	Finally, observe that almost all solvable groups of depth at most $d$ actually have depth precisely $d$, i.e. almost every element in $\g^s_{\leq d}$ is actually contained in the subset $\g^s_d$. Therefore 
	\[
		\lim_{d \to \infty} \lim_{n \to \infty} \dfrac{\sum_{i=1}^n P(\g^s_{\leq d}(i))}{\sum_{i=1}^n T(\g^s_{\leq d}(i))} = \lim_{d \to \infty} \lim_{n \to \infty} \dfrac{\sum_{i=1}^n P(\g^s_d(i))}{\sum_{i=1}^n T(\g^s_d(i))}.
	\]
	By \cref{Prop:DepthBoundG,Lem:Sums}, we have
	\[
		\lim_{d \to \infty} \lim_{n \to \infty} \dfrac{\sum_{i=1}^n P(\g^s_d(i))}{\sum_{i=1}^n T(\g^s_d(i))} \leq \lim_{d \to \infty} \lim_{n \to \infty} \max_n \dfrac{ P(\g^s_d(n))}{T(\g^s_d(n))} \leq \lim_{d \to \infty} \lim_{n \to \infty}\max_n \dfrac{1}{d} = 0,
	\]
	which implies the result in view of the previous equalities. 
\end{proof}

\begin{remark}
	In the proof of \cref{Thm:Asymptotics}, we used the fact that almost all finite groups are solvable to reduce our analysis to the solvable case. If one restricts further to cyclic groups of prime power order, one can obtain explicit bounds using the classification of indexing categories for such groups by Balchin--Barnes--Roitzheim \cite{BBR19}. In fact, the tables in \cref{App:Tables} suggested the more general statement of \cref{Thm:Asymptotics}. 
\end{remark}

\section{Freeness after Localization}\label{Sec:FreeAfterLocalization}

In classical algebra, modules often become free when the base ring is enlarged. For example, base-change from $\z$ to $\q$ (or any field) forces any module to be free. In this section, we explore what happens to free incomplete Tambara functors after inverting various elements in the Burnside functor. 

\subsection{Cohomological and Fixed Point Mackey functors}\label{SS:Cohomological}
Before discussing localization, we discuss two important classes of Mackey functors: cohomological Mackey functors and fixed point Mackey functors. 

\begin{definition}[{\cite[Section 1.4]{Gre71}}]
	A Mackey functor $\uM$ is \emph{cohomological} if
	\[
		\tr_K^H \res_K^H(x) = [H:K]x
	\]
	for all $x \in \uM(G/H)$ and all subgroups $K \leq H \leq G$. 
\end{definition}

The name ``cohomological" comes from the analogous relations satisfied by restriction and transfer in group cohomology. Cohomological Mackey functors have a coordinate free description as well. 

\begin{proposition}[{\cite[Proposition 16.3]{TW95}}]
	There is an equivalence between the category of $\uZ$-modules and the full subcategory of cohomological Mackey functors.  
\end{proposition}

In particular, being a $\uZ$-module is a property of a Mackey functor, rather than extra structure. A Mackey functor $\uM$ is a $\uZ$-module if and only if the homomorphism 
\[
	M \cong \uA \boxtimes \uM \to \uZ \boxtimes \uM
\]
induced by the unit  $\eta : \uA \to \uZ$ is an isomorphism. 

\begin{remark}\label{Rmk:ZQuotient}
	Base change $\uZ \boxtimes -$ from $\uA$ to $\uZ$ may be calculated levelwise by the formula
	\[
		(\uZ \boxtimes \uM)(G/H) \cong \uM(G/H)/([H:K]x - \tr_K^H \res_K^H x).
	\]
	This follows from the fact that $\uZ$ is a quotient of $\uA$ and the box product commutes with colimits in each variable. 
\end{remark}

We now aim to characterize the free cohomological Mackey functors. To do so, we must first introduce another class of Mackey functors. 

\begin{definition}\label{exm:FixedPointFunctor}
	Let $\FP$ be the right adjoint to the forgetful functor from Mackey functors to $G$-modules which sends $\uM$ to $\uM(G/e)$ with its Weyl group action. If $V$ is a $G$-module, then $\FP(V)$ is called its \emph{fixed point Mackey functor}. 

	If $k$ is a $G$-ring, then $\FP(k)$ is a Tambara functor, and we call it a \emph{fixed point Tambara functor} of $k$. Forgetting norms, we obtain \emph{fixed point $\co$-Tambara functors}.
\end{definition}

Explicitly, if $V$ is a $G$-module, its fixed point Mackey functor $\FP(V)$ is given at level $G/H$ by $\FP(V)(G/H) = V^H$. Restriction is given by inclusion of fixed points, transfer is given by additive transfer, and norm is given by multiplicative transfer.

\begin{example}\label{exm:ConstantTambara}
	Consider $\z$ equipped with a trivial $G$-action. The fixed point Tambara functor of this $G$-ring is the constant Tambara functor $\uZ$.
\end{example}

\begin{example}\label{Ex:FreeZisFixedPoint}
	Let $H \leq G$. There is an isomorphism of $\uZ$-modules
	\[
		\uZ\{x_{G/H}\} \cong \FP(\z[G/H])
	\]
	between the free $\uZ$-module on a generator at level $G/H$ and the fixed point Mackey functor (\cref{exm:FixedPointFunctor}) of the permutation $G$-module $\z[G/H]$. 
\end{example}

\begin{remark}
	Any Mackey functor $\uM$ has a canonical homomorphism
	\[
		\uM \to \FP(\uM(G/e)).
	\]
	The right-hand side is a $\uZ$-module, so this homomorphism factors through the base change to $\uZ$ to give another
	\[
		\uZ \boxtimes \uM \to \FP(\uM(G/e)).
	\]
\end{remark}

\begin{proposition}\label{Theorem:FreeCohomological}
	A free cohomological Mackey functor is of the form $\FP(V)$ where $V$ is a permutation $G$-module. 
\end{proposition}

\begin{proof}
	Any free cohomological Mackey functor $\uM$ is of the form 
	\[ 
		\uM = \bigoplus_{i \in I} \uZ\{x_{T_i}\} 
	\]
	where each $T_i$ is a finite $G$-set. Decomposing these finite $G$-sets into orbits, we may write $\uM$ as a sum of transitive finite $G$-sets:
	\[ 
		\uM \cong \bigoplus_{j \in J} \uZ\{x_{G/H_j}\}.
	\]
	From \cref{Ex:FreeZisFixedPoint}, each summand $\uZ\{x_{G/H_j}\}$ is a fixed point functor $\FP(\z[G/H_j])$. The functor $\FP$ commutes with direct sums \cite[Proposition 2.3]{TW95}, and direct sums of permutation modules are again permutation modules. So we are done.
\end{proof}

\begin{remark}
	Thevenaz--Webb prove in \cite[Theorem 16.5]{TW95} that every cohomological Mackey functor is a quotient of a fixed point functor $\FP(V)$ for some permutation $G$-module $V$, but they do not explicitly describe the free objects. 
\end{remark}

\subsection{Localization}\label{SS:localization}

We now discuss localization for incomplete Tambara functors. 

\begin{definition}[{\cite[Definition 5.21]{BH18}}]\label{Def:Localization}
	Let $\uR$ be an $\co$-Tambara functor and let $\uS = \{ (a_i, T_i) \mid a_i \in \uR(T_i), \ i \in I\}$ be a collection of elements in the values of $\uR$ at various finite $G$-sets. 

	A map $\phi: \uR \to \uR'$ of $\co$-Tambara functors \emph{inverts} $\uS$ if for all $i \in I$, $\phi(a_i) \in \uR'(T_i)^{\times}.$

	Let $\phi : \uR \to \uS^{-1}\uR$ be the initial map of $\co$-Tambara functors which inverts $\uS$.\footnote{This exists by \cite[Theorem 5.23]{BH18}.} We will refer to $\uS^{-1}\uR$ as the \emph{localization of $\uR$ at $\uS$}. 

	If $\uM$ is an $\uR$-module, define $S^{-1}\uM := \uS^{-1} \uR\boxtimes_{\uR} \uM$. 
\end{definition}

We consider several important localizations of the Burnside $\co$-Tambara functor.

\begin{example}
	Consider the localization of the Burnside Tambara functor $\uA$ at $\uS = \{(n, G/G) \mid n \in \mathbb N_{>0}\}$. The Tambara functor $\uS^{-1}\uA$ is the rational Burnside functor \(\uA_{\q}\). On a finite $G$-set $T$, $\uA_\q(T) \cong \uA(T) \otimes_\z \q$. A module over $\uA_\q$ is a \emph{rational Mackey functor}, i.e. a Mackey functor $\uM$ such that each $\uM(T)$ is a rational vector space. 
\end{example}

\begin{example}
	Consider the Burnside Mackey functor $\uA$. Let $\uS = \{(|G|, G/G)\}$, and write $\uA[\tfrac{1}{|G|}] := \uS^{-1}\uA$. On a finite $G$-set $T$, this Mackey functor is given by 
	\[
		\uA[\tfrac{1}{|G|}](T) = \uA(T)[\tfrac{1}{|G|}] = \uA(T) \otimes_\z \z[\tfrac{1}{|G|}].
	\]
	Modules over this Mackey functor are Mackey functors taking values in $\z[\tfrac{1}{|G|}]$-modules rather than abelian groups. 
\end{example}

Instead of inverting the order of $G$, we could categorify and invert the class of $G$ as a finite $G$-set in the Burnside ring $\uA(G/G)$ of finite $G$-sets. 

\begin{notation}
	Write $\uA[\tfrac{1}{[G/e]}] := \uS^{-1}\uA$ for $\uS = \{([G], G/G)\}$. 
\end{notation}

%
%
%

\begin{lemma}\label{lemma:invertGInvertIndex}
	For all $K \leq H \leq G$, $[H/K]$ and $[H : K]$ are invertible and equal in $\uA[\tfrac{1}{[G/e]}](G/H)$. 
\end{lemma}

\begin{proof}
	We show this first at level \(G/G\). By the Frobenius relation 
	\[ 
	    \tr^H_K(a) \cdot b=\tr_K^H(a \cdot \res^H_K(b)), 
	\]
	we obtain equations
	\[
		[G/\{e\}]\cdot [G/H]=[G/\{e\}]\cdot [G:H].
	\]
	in \(\uA(G/G)\). In the localization inverting \([G/\{e\}]\), we deduce that for all \(H\),
	\[
	[G/H]=[G:H].
	\]
	When \(H=\{e\}\), this implies that \([G/\{e\}]=|G|\) and hence that \(|G|\) is a unit. Since \([G:H]\) divides \(|G|\), it is a unit, and hence for all \(H\), \([G/H]\) is a unit.
	
	For an arbitrary level \(G/H\), note that 
	\[
	i_H^\ast G/\{e\}=[G:H] H/\{e\},
	\]
	and hence inverting \([G/\{e\}]\) at level \(G/G\) also inverts \([H/\{e\}]\). The result then follows from the analysis for \(G/G\).
\end{proof}

Consider the homomorphism of Tambara functors $\uA \to \uZ[\tfrac{1}{|G|}]$ given by the composite of $\uA \to \uZ$ and localization. At level $G/H$, a finite $H$-set is sent to its cardinality. In particular, the image of $[G/e] \in \uA(G/G)$ is a unit. 
Hence we obtain a homomorphism of Tambara functors 
\[ 
	\alpha \colon \uA[\tfrac{1}{[G/e]}] \to \uZ[\tfrac{1}{|G|}] 
\] 
from the universal property of localization. 

\begin{theorem}\label{prop:invertGiso}
	The morphism $\alpha \colon \uA[\tfrac{1}{[G/e]}] \to \uZ[\tfrac{1}{|G|}]$ is an isomorphism of Tambara functors.
\end{theorem}

\begin{proof}
Since \([G/\{e\}]\) being a unit implies that \(|G|\) is a unit, and since all restriction maps are ring homomorphisms, both the source and the target of the universal map are actually Mackey functors in \(\z[1/|G|]\)-modules. 

As a Mackey functor, the constant Mackey functor \(\uZ[1/|G|]\) is generated by the element \(1\) at level \(G/G\), and the element \(1\) is in the image of the map from \(\uA\), and hence the localization. This means that the natural map is surjective.

\cref{lemma:invertGInvertIndex} shows that for any subgroup \(H\), the map
\[
\uA\big[\tfrac{1}{|G|}\big]\to\uA\big[\tfrac{1}{[G/e]}\big](G/H)
\]
factors through the quotient by the ideal generated by \([H/K]-[H:K]\) for all \(K\). This is the contant Mackey functor \(\z[1/|G|]\), and hence the map is an isomorphism. 
\end{proof}

\subsection{Modules over $\uA[\tfrac{1}{|G|}]$ and $\uA[\tfrac{1}{[G/e]}]$}\label{SS:splitting}

By \cref{prop:invertGiso}, modules over $\uA[\tfrac{1}{[G/e]}]$ and modules over ${\uZ[\frac{1}{|G|}]}$ are equivalent. By this observation, $\uA[\tfrac{1}{[G/e]}]$-modules are cohomological $\uA[\tfrac{1}{|G|}]$-modules. We aim to characterize such modules. 

We begin by discussing the splitting of the category of rational Mackey functors. Let $\Sub_c(G)$ denote a set of representatives for conjugacy classes of subgroups of $G$.

Work of Dress \cite{Dre69} constructs a system of orthogonal idempotents for the rational Burnside ring. For each subgroup $H$ of $G$, define a \emph{mark homomorphism} $\phi_H \colon \uA(G/G) \to \z$ by $\phi_H([T]) = |T^H|$. Rationally, these assemble into an isomorphism of rings 
	\[
		\phi \colon \uA_\q(G/G) \xrightarrow{\cong} \prod_{H \in \Sub_c(G)} \q. 
	\]
	The projections onto individual factors yield a system of orthogonal idempotents $\{e_H\}_{H \in \Sub_c(G)}$ for $\uA_\q(G/G)$ characterized by the property that 
	\[ 
		\phi_H(e_K) = 
		\begin{cases}
			1 & H \text{ and } K \text{ are conjugate in } G,\\
			0 & \text{otherwise.}
		\end{cases}
	\]

\begin{example}\label{Example:IdempotentForTrivial}
	For $H = e$, $e_{\{e\}} = \frac{1}{|G|}[G/e] \in \uA_\q(G/G)$ is the desired idempotent. Indeed, one may verify that $e_{\{e\}}^2 = e_{\{e\}}$, and  
	\[
		\phi_H\left(\tfrac{1}{|G|}[G/e]\right) = \tfrac{1}{|G|}|G^H| = 
		\begin{cases}
			1 & H = e,\\
			0 & H \neq e. 
		\end{cases}
	\]
\end{example}

Since $\uA(G/G) \cong \Mack_G(\uA, \uA)$, this system of orthogonal idempotents gives a splitting of the rational Burnside Mackey functor. We summarize these results in the following theorem:

\begin{theorem}[{\cite{Dre69}}]\label{Thm:RationalBurnsideFunctorSplits}
	There is a system of orthogonal idempotents $\{e_H\}_{H \in \Sub_c(G)}$ splitting the rational Burnside Green functor: there is an isomorphism of Green functors
	\[ 
		\uA_\q \cong \bigoplus_{H \in \Sub_c(G)} e_H \uA_\q, 
	\]
	where $e_H \uA_\q$ is the sub-Green functor of $\uA_\q$ with
	\[ 
		(e_H\uA_\q)(G/K) = \res^G_K(e_H) \cdot \uA_\q(G/K).
	\]
\end{theorem}

This splitting of the monoidal unit gives a canonical splitting of the category of rational Mackey functors. Greenlees--May \cite[Appendix A]{GM95} then prove equivalences between $e_H\uA_\q$-modules and modules over a rational group-ring.

\begin{theorem}[{\cite[Theorem A.9]{GM95}}]\label{Thm:RationalMackeyCategorySplits}
	For any $H \leq G$ there is an equivalence of categories
	\[
		U_H \colon e_H\uA_\q\mhyphen\Mod  \simeq \displaystyle \q[W_G(H)]\mhyphen\Mod \colon F_H
	\]
	Both $U_H$ and $F_H$ are exact functors. 
\end{theorem}

Together, the previous two theorems combine to give an equivalence of categories 
\[ 
	\uA_\q\mhyphen\Mod \simeq \prod_{H \in \Sub_c(G)} \q[W_G(H)]\mhyphen\Mod. 
\]
This equivalence was independently proven in \cite[Appendix A]{GM95} and \cite{TW95} using very different methods -- the former approaches the problem from the perspective of stable homotopy theory, whereas the latter uses algebraic techniques. A recent exposition of this result can be found in \cite{BK20}.  

A consequence of this theorem is that all rational Mackey functors are projective. Indeed, by Maschke's theorem all $\q[W_G(H)]$-modules are projective, and the equivalence of the previous theorem is by exact functors. 

\begin{corollary}[{\cite[Proposition A.2]{GM95}}]
	All rational Mackey functors are projective. 
\end{corollary}

To prove \cref{Thm:RationalMackeyCategorySplits,Thm:RationalBurnsideFunctorSplits}, it turns out that it isn't necessary to work rationally, but merely to invert $|G|$. 

\begin{theorem}\label{theorem:invertOrderSplitsModules}
	There is an equivalence of categories 
	\begin{equation}\label{Eq:InvertOrderSplitsModules}
		U \colon \uA[\tfrac{1}{|G|}]\mhyphen\Mod \simeq \prod_{H \in \Sub_c(G)} \z[\tfrac{1}{|G|}][W_G(H)]\mhyphen\Mod \colon F.
	\end{equation}
	Both $U$ and $F$ are exact functors. 
\end{theorem}

\begin{proof}
	By a careful analysis of the proofs of \cref{Thm:RationalBurnsideFunctorSplits,Thm:RationalMackeyCategorySplits} in \cite{BK20}, there are only two places where it is necessary to invert an integer: in the construction of orthogonal idempotents splitting the Burnside ring in \cite[Lemma 2.2]{BK20} and the isomorphism between orbits and fixed points in \cite[Example 2.9]{BK20}. In both places, it suffices to invert the orders of all subgroups of $G$. Since inverting $|G|$ necessarily inverts all of its divisors, it is sufficient to work in $\z[\tfrac{1}{|G|}]$ instead of $\q$. 
\end{proof}

From this theorem, we will discover that the summand of $\uA[\frac{1}{|G|}]$ corresponding to the trivial subgroup is $\uA[\frac{1}{[G/e]}]$, and therefore $\uA[\frac{1}{[G/e]}]$-modules are equivalent to modules over a particular group-ring.   

\begin{lemma}\label{Lemma:TheTrivialSummand}
	There is an isomorphism of Green functors 
	$e_{\{e\}}\uA[\tfrac{1}{|G|}] \cong \uZ[\tfrac{1}{|G|}]$.
\end{lemma}

\begin{proof}
	At level $G/H$, we have 
	\[
		\left(e_{\{e\}}\uA[\tfrac{1}{|G|}]\right)(G/H) = \res^G_e(e_{\{e\}}) \cdot \uA[\tfrac{1}{|G|}](G/H).
	\]
	Recall from \cref{Example:IdempotentForTrivial} that $e_{\{e\}} = \frac{1}{|G|}[G/e]$. Therefore, 
	\[
		\res^G_H(e_{\{e\}}) = \tfrac{1}{|H|}[H/e].
	\]

	So at level $G/H$, $e_{\{e\}}\uA[\tfrac{1}{|G|}]$ is the Green ideal of $\uA[\tfrac{1}{|G|}](G/H)$ generated by $[H/e]$. In particular, at the underlying level
	\[
		(e_{\{e\}}\uA[\tfrac{1}{|G|}])(G/e) = \uA[\tfrac{1}{|G|}](G/e) = \z[\tfrac{1}{|G|}].
	\]
	Therefore, every element in $e_{\{e\}}\uA[\tfrac{1}{|G|}]$ at level $G/H$ is a transfer of an element at level $G/e$. We conclude that $e_{\{e\}}\uA[\tfrac{1}{|G|}]$ is the sub-Mackey functor of $\uA[\tfrac{1}{|G|}]$ generated by the underlying level, that is, 
	\[
		e_{\{e\}}\uA[\tfrac{1}{|G|}] \cong \uZ[\tfrac{1}{|G|}]. 
	\]
	Since each level is a commutative ring, and we have levelwise isomorphisms of commutative rings, we conclude this is an isomorphism of Green functors. 
\end{proof}

We also need a description of the functors $U_H$ and $F_H$ when $H$ is the trivial subgroup. 

\begin{lemma}[{cf. \cite[Proposition 4.5]{BK20}}]\label{Lemma:EquivalenceForTrivialSubgroup}
	When $H = \{e\}$, the equivalence 
	\[
		U_{\{e\}} \colon  e_{\{e\}}\uA_\q\mhyphen\Mod \simeq \q[G]\mhyphen\Mod \colon F_{\{e\}} 
	\]
	is given by functors 
	\[
		U_{\{e\}}(\uM) := \uM(G/\{e\}) \text{ and } F_{\{e\}}(V) := \FP(G). 
	\]
\end{lemma}

\begin{corollary}
	There is an equivalence of categories 
	\[
	\begin{tikzcd}
		\uA[\tfrac{1}{[G/e]}]\mhyphen\Mod 
			\arrow[yshift=2]{r}{U}
			& 
		\z[\tfrac{1}{|G|}][G]\mhyphen\Mod, 
			\arrow[yshift=-2]{l}{\FP}
	\end{tikzcd}
	\]
	where $U(\uM) = \uM(G/e)$ and $\FP$ is the fixed point functor, such that:
	\begin{enumerate}[(a)]
		\item both $U$ and $\FP$ are exact functors;
		\item both $U$ and $\FP$ are strong symmetric monoidal for the box-product over $\uA[\tfrac{1}{[G/e]}]$ on $\uA[\tfrac{1}{[G/e]}]\mhyphen\Mod$ and the tensor product over $\z[\tfrac{1}{|G|}]$ on $\z[\tfrac{1}{|G|}][G]\mhyphen\Mod$.
	\end{enumerate}
\end{corollary}

\begin{proof}
	By \cref{theorem:invertOrderSplitsModules}, there is an equivalence of categories 
	\[
		e_{\{e\}}\uA[\tfrac{1}{|G|}]\mhyphen\Mod \simeq \z[\tfrac{1}{|G|}][G]\mhyphen\Mod, 
	\]
	which by \cref{Lemma:EquivalenceForTrivialSubgroup} is given by functors $U$ and $\FP$ as in the statement of the corollary.
	By \cref{prop:invertGiso,Lemma:TheTrivialSummand} there are isomorphisms of Green functors 
	\[
		e_{\{e\}}\uA[\tfrac{1}{|G|}] \cong \uZ[\tfrac{1}{|G|}] \cong \uA[\tfrac{1}{[G/e]}], 
	\]
	yielding the desired equivalence of categories. 

	Exactness follows from exactness in \cref{theorem:invertOrderSplitsModules}. The strong symmetric monoidal property follows for $U$ because 
	\[
		U(\uA[\tfrac{1}{[G/e]}]) = \z[\tfrac{1}{|G|}] \quad \text{and} \quad U(\uM \boxtimes \uN) = \uM(G/e) \otimes_\z \uN(G/e);
	\] 
	the box product/tensor over the localization is the same as the ordinary box/tensor product. Then $\FP$ becomes strong symmetric monoidal as part of an equivalence. 
\end{proof}

\begin{corollary}\label{cor:invertGAllModulesAreFixedPoint}
	Any $\uA[\tfrac{1}{[G/e]}]$-module $\uM$ is a fixed point functor: $\uM \cong \FP(\uM(G/e))$. 
\end{corollary}

\begin{remark}
	In light of this corollary, every restriction in any $\uA[\tfrac{1}{[G/e]}]$-module $\uM$ is injective -- it is the inclusion of fixed points. The condition that all restrictions are injective appears in several seemingly unrelated places. This is the \emph{monomorphic restriction condition} of \cite[Definition 4.19]{Nak12}. It is also the condition necessary for a Mackey functor to be a zero-slice of an equivariant spectrum \cite[Proposition 4.50]{HHR16}. This seemingly innocuous condition has many structural consequences for Mackey functors. In general any such functor satisfying the monomorphic restriction condition is a sub-functor of a fixed point functor, by \cite[Proposition 4.21]{Nak12}.
\end{remark}

\subsection{Underlying freeness after localization}
\label{SS:freeLocal}

We prove that all free incomplete Tambara functors over $\uA[\tfrac{1}{[G/e]}]$ are free as $\uA[\tfrac{1}{[G/e]}]$-modules. We will write $\uS^{-1}\uA = \uA[\tfrac{1}{[G/e]}]$, with $\uS = \{([G/e], G/G)\}$ to declutter notation. 

\begin{lemma}\label{Lemma:InvertGUnderlyingLevelIsPermutationModule}
	Let $\co$ be any indexing category, and let $H \leq G$ be a subgroup. Then as a $\z[\tfrac{1}{|G|}]$-algebra, $\uS^{-1}\uA^\co[x_{G/H}](G/e)$ is polynomial on generators $y_{gH}$ for cosets $gH \in G/H$: 
	\[
		\uS^{-1}\uA^\co[x_{G/H}](G/e) \cong \z[\tfrac{1}{|G|}]\big[y_{gH}\mid gH \in G/H\big]. 
	\]
	$G$ acts on the generators by permuting the cosets. In particular, $\uS^{-1}\uA^\co[x_{G/H}](G/e)$ is a permutation $\z[\tfrac{1}{|G|}]$-module.
\end{lemma}

\begin{proof}
	We have 
	\[
		\uS^{-1}\uA^\co[x_{G/H}](G/e) = (\uS^{-1}\uA \boxtimes \uA^{\co}[x_{G/H}])(G/e) = \cp^G_{\co}(G/H, G/e)^+ \otimes_\z \z[\tfrac{1}{|G|}],
	\]
	so it suffices to describe the ring $\cp^G_{\co}(G/H, G/e)^+$. 

	Consider the polynomials of the form
	\[ 
		y_g = [G/H \xleftarrow \pi G/e \xrightarrow g G/e \xrightarrow \id G/e], 
	\]
	where $\pi$ is the canonical projection and $g$ is multiplication by $g \in G$. Note that the middle map is always admissible, for any indexing category $\co$. Two such basic polynomials are equivalent if the middle map differs by an element of $H$, so we have one such basic polynomial for each coset of $G/H$. 

	We claim that $\{y_g\}$ is a generating set for $\cp^G_{\co}(G/H, G/e)$ as a ring, where $g$ runs over a set of coset representatives for $G/H$. Recall the ring structure from \cref{Thm:Tam93}. 

	Given any polynomial 
	\[
		\Sigma = [G/H \xleftarrow f A \xrightarrow g B \xrightarrow h G/e], 
	\]
	first observe that we may write 
	\[
		\Sigma = [G/H \xleftarrow f A \xrightarrow g \im(g) \xrightarrow h G/e] + [G/H \xleftarrow f \emptyset \xrightarrow g (B \setminus \im(g)) \xrightarrow h G/e], 
	\]
	and this second summand is a sum of copies of $1 = [G/H \leftarrow \emptyset \to G/e \xrightarrow \id G/e]$. So we may assume that $g$ is surjective. 

	Second, we may decompose 
	\[
		B = \bigsqcup_{i \in I} B_i, \quad A = \bigsqcup_{i \in I} g^{-1}(B_i)
	\]
	with $B_i \cong G/e$ for all $i$. Then 
	\[
		\Sigma = \sum_{i \in I} [G/H \xleftarrow f g^{-1}(B_i) \xrightarrow g B_i \xrightarrow h G/e].
	\]
	So we may further reduce to the case where $B \cong G/e$. 

	We are left with $\Sigma = [G/H \leftarrow A \twoheadrightarrow G/e \to G/e]$. Decompose 
	$A = \bigsqcup_{j \in J} A_j$
	with $A_j \cong G/e$ for all $j$. Then 
	\[
		\Sigma = \prod_{j \in J} [G/H \leftarrow A_j \to G/e \to G/e] 
	\]
	is a product of polynomials of the form $y_g$. This establishes the claim. 

	Finally, this polynomial ring is a permutation module because the $G$-action permutes the monomials.
\end{proof}

\begin{remark}
	The argument above is a stronger form of that used to prove \cref{cor:PolynomialsAsPolynomials}. 
\end{remark}


\begin{theorem}\label{Theorem:FreeLocalUnderlyingFree}
	Let $\underline{S}^{-1}\uA$ be the Mackey functor obtained from $\uA$ by inverting $[G/e] \in \uA(G/G)$. For any indexing category $\co$ and subgroup $H \leq G$, the free $\co$-Tambara functor $\uS^{-1}\uA^{\co}[x_{G/H}]$ is free as an $\uS^{-1}\uA$-module. 
\end{theorem}

\begin{proof}
	By \cref{Lemma:InvertGUnderlyingLevelIsPermutationModule}, the underlying level of $\uS^{-1}\uA^\co[x_{G/H}]$ is a permutation module. By \cref{cor:invertGAllModulesAreFixedPoint}, $\uS^{-1}\uA^\co[x_{G/H}]$ is isomorphic to the fixed points of a permutation module as an $\uS^{-1}\uA$-module. In particular, it is a free module by \cref{Theorem:FreeCohomological}.
\end{proof}

\cref{Lemma:InvertGUnderlyingLevelIsPermutationModule} also has another interesting consequence. 

\begin{proposition}
	Any two free incomplete $\uS^{-1}\uA$-Tambara functors generated at the same level $G/H$ are isomorphic as Green functors. 
\end{proposition}

\begin{proof}
	By \cref{Lemma:InvertGUnderlyingLevelIsPermutationModule}, they are both isomorphic to the fixed-point functor of the ring $\z[\tfrac{1}{|G|}][y_{gH} \mid gH \in G/H]$.
\end{proof}

Essentially, the only difference between free incomplete $\uS^{-1}\uA$-Tambara functors is the norms.

\appendix

\section{Tables of Underlying Mackey Functors}\label{App:Tables}

The tables below describe the Mackey functors underlying free $\co$-Tambara functors generated by a single element at level $G/H$ for various combinations $G$ and $H$. The columns indicate the level of the generator and the rows indicate the indexing category. An indexing category $\co$ for $G$ is represented by a graph whose vertices are the subgroups of $G$ and edges $K \to H$ indicates that $H/K$ is admissible for $i^*_H \co$.\footnote{For the reader familiar with \cite{Rub20}, the diagrams in the left-hand column are just transfer systems.} Within each cell of the table, the freeness of the underlying Mackey functor is stated, with the convention that an empty cell indicates that the Mackey functor underlying is not free. 

The values in the tables are deduced from \cref{Thm:SolvableCase}. The classifications of all of the indexing categories for cyclic groups below appear in \cite[Theorem 2]{BBR19} and \cite[Section 3.2]{Rub20}. The classification of categories for $D_6$ appears in \cite[Section 3.2]{Rub20}.

\begin{table}[h]
\label{Table:Cp}
\caption{This table describes the underlying Mackey functor of $\uA^{\co}[x_{C_p/H}]$. The table says, for example, that $\uA^{\co^{\cplt}}[x_{C_p/e}]$ is free as an $\uA$-module, while $\uA^{\co^{\cplt}}[x_{C_p/C_p}]$ is not free as an $\uA$-module.}
\centering
\[ \begin{array}{|c||c|c|}
\hline
 & C_p/e & C_p/C_p \\ \hline \hline
\co^\triv &  & \text{free} \\ \hline 
\co^\cplt & \text{free} & \\ \hline
\end{array} \]
\end{table}

\begin{table}[h]
\label{Table:Cp2}
\caption{The table below describes the Mackey functor underlying $\uA^{\co}[x_{C_{p^2}/H}]$.}
\[ \begin{array}{|c||c|c|c|}
\hline
 & C_{p^2} / e & C_{p^2} / C_p & C_{p^2} / C_{p^2} \\ \hline \hline
\begin{tikzpicture}
\draw[fill] (0,0) node(e) {$e$};
\draw[fill] (1,0) node(Cp) {$C_p$};
\draw[fill] (2,0) node(Cp2) {$C_{p^2}$};
\end{tikzpicture}
&  &  & \text{free} \\ \hline
\begin{tikzpicture}
\draw[fill] (0,0) node(e) {$e$};
\draw[fill] (1,0) node(Cp) {$C_p$};
\draw[fill] (2,0) node(Cp2) {$C_{p^2}$};
\draw[->] (e) -- (Cp);
\end{tikzpicture}
&  &   &  \\ \hline
\begin{tikzpicture}
\draw[fill] (0,0) node(e) {$e$};
\draw[fill] (1,0) node(Cp) {$C_p$};
\draw[fill] (2,0) node(Cp2) {$C_{p^2}$};
\draw[->] (e) -- (Cp);
\draw[->] (e) to[bend left] (Cp2);
\end{tikzpicture}
& \text{free} &   &  \\ \hline
\begin{tikzpicture}
\draw[fill] (0,0) node(e) {$e$};
\draw[fill] (1,0) node(Cp) {$C_p$};
\draw[fill] (2,0) node(Cp2) {$C_{p^2}$};
\draw[->] (Cp) -- (Cp2);
\end{tikzpicture}
&  & \text{free}  &  \\ \hline
\begin{tikzpicture}
\draw[fill] (0,0) node(e) {$e$};
\draw[fill] (1,0) node(Cp) {$C_p$};
\draw[fill] (2,0) node(Cp2) {$C_{p^2}$};
\draw[->] (e) -- (Cp);
\draw[->] (e) to[bend left] (Cp2);
\draw[->] (Cp) -- (Cp2);
\end{tikzpicture}
& \text{free} &   & \\ \hline
\end{array} \]
\end{table}

\begin{table}[h!]
\label{Table:Cp3}
\caption{The table below describes the Mackey functor underlying $\uA^{\co}[x_{C_{p^3}/H}]$.}
\[ \begin{array}{|c||c|c|c|c|}
\hline
& C_{p^3}/e & C_{p^3}/C_p & C_{p^3} / C_{p^2} & C_{p^3}/C_{p^3} \\ \hline \hline 
\begin{tikzpicture}
\draw (0,0) node(e) {$e$};
\draw (1,0) node(Cp) {$C_p$};
\draw (2,0) node(Cp2) {$C_{p^2}$};
\draw (3,0) node (Cp3) {$C_{p^3}$};
\end{tikzpicture}
& & & & \text{free} \\ \hline
\begin{tikzpicture}
\draw (0,0) node(e) {$e$};
\draw (1,0) node(Cp) {$C_p$};
\draw (2,0) node(Cp2) {$C_{p^2}$};
\draw (3,0) node (Cp3) {$C_{p^3}$};
\draw[->] (e) -- (Cp);
\end{tikzpicture}
& & & & \\ \hline
\begin{tikzpicture}
\draw (0,0) node(e) {$e$};
\draw (1,0) node(Cp) {$C_p$};
\draw (2,0) node(Cp2) {$C_{p^2}$};
\draw (3,0) node (Cp3) {$C_{p^3}$};
\draw[->] (Cp) -- (Cp2);
\end{tikzpicture}
& & &  & \\ \hline
\begin{tikzpicture}
\draw (0,0) node(e) {$e$};
\draw (1,0) node(Cp) {$C_p$};
\draw (2,0) node(Cp2) {$C_{p^2}$};
\draw (3,0) node (Cp3) {$C_{p^3}$};
\draw[->] (Cp2) -- (Cp3);
\end{tikzpicture}
& & & \text{free} & \\ \hline
\begin{tikzpicture}
\draw (0,0) node(e) {$e$};
\draw (1,0) node(Cp) {$C_p$};
\draw (2,0) node(Cp2) {$C_{p^2}$};
\draw (3,0) node (Cp3) {$C_{p^3}$};
\draw[->] (e) -- (Cp);
\draw[->] (Cp2) -- (Cp3);
\end{tikzpicture}
& &  & & \\ \hline
\begin{tikzpicture}
\draw (0,0) node(e) {$e$};
\draw (1,0) node(Cp) {$C_p$};
\draw (2,0) node(Cp2) {$C_{p^2}$};
\draw (3,0) node (Cp3) {$C_{p^3}$};
\draw[->] (e) -- (Cp);
\draw[->] (e) to[bend left] (Cp2);
\end{tikzpicture}
&  &  & & \\ \hline
\begin{tikzpicture}
\draw (0,0) node(e) {$e$};
\draw (1,0) node(Cp) {$C_p$};
\draw (2,0) node(Cp2) {$C_{p^2}$};
\draw (3,0) node (Cp3) {$C_{p^3}$};
\draw[->] (Cp) -- (Cp2);
\draw[->] (Cp) to[bend left] (Cp3);
\end{tikzpicture}
& & \text{free} & & \\ \hline
\begin{tikzpicture}
\draw (0,0) node(e) {$e$};
\draw (1,0) node(Cp) {$C_p$};
\draw (2,0) node(Cp2) {$C_{p^2}$};
\draw (3,0) node (Cp3) {$C_{p^3}$};
\draw[->] (e) -- (Cp);
\draw[->] (e) to[bend left] (Cp2);
\draw[->] (e) to[bend left] (Cp3);
\end{tikzpicture}
&  \text{free} &  & & \\ \hline
\begin{tikzpicture}
\draw (0,0) node(e) {$e$};
\draw (1,0) node(Cp) {$C_p$};
\draw (2,0) node(Cp2) {$C_{p^2}$};
\draw (3,0) node (Cp3) {$C_{p^3}$};
\draw[->] (e) -- (Cp);
\draw[->] (Cp) -- (Cp2);
\draw[->] (e) to[bend left] (Cp2);
\end{tikzpicture}
& & & & \\ \hline
\begin{tikzpicture}
\draw (0,0) node(e) {$e$};
\draw (1,0) node(Cp) {$C_p$};
\draw (2,0) node(Cp2) {$C_{p^2}$};
\draw (3,0) node (Cp3) {$C_{p^3}$};
\draw[->] (Cp) -- (Cp2);
\draw[->] (Cp2) -- (Cp3);
\draw[->] (Cp) to[bend left] (Cp3);
\end{tikzpicture}
& & \text{free} & & \\ \hline
\begin{tikzpicture}
\draw (0,0) node(e) {$e$};
\draw (1,0) node(Cp) {$C_p$};
\draw (2,0) node(Cp2) {$C_{p^2}$};
\draw (3,0) node (Cp3) {$C_{p^3}$};
\draw[->] (e) -- (Cp);
\draw[->] (Cp) -- (Cp2);
\draw[->] (e) to[bend left] (Cp2);
\draw[->] (e) to[bend left] (Cp3);
\end{tikzpicture}
& \text{free}& & & \\ \hline
\begin{tikzpicture}
\draw (0,0) node(e) {$e$};
\draw (1,0) node(Cp) {$C_p$};
\draw (2,0) node(Cp2) {$C_{p^2}$};
\draw (3,0) node (Cp3) {$C_{p^3}$};
\draw[->] (e) -- (Cp);
\draw[->] (Cp2) -- (Cp3);
\draw[->] (e) to[bend left] (Cp2);
\draw[->] (e) to[bend left] (Cp3);
\end{tikzpicture}
& \text{free}& & & \\ \hline
\begin{tikzpicture}
\draw (0,0) node(e) {$e$};
\draw (1,0) node(Cp) {$C_p$};
\draw (2,0) node(Cp2) {$C_{p^2}$};
\draw (3,0) node (Cp3) {$C_{p^3}$};
\draw[->] (e) -- (Cp);
\draw[->] (Cp) -- (Cp2);
\draw[->] (e) to[bend left] (Cp2);
\draw[->] (e) to[bend left] (Cp3);
\draw[->] (Cp) to[bend right] (Cp3);
\end{tikzpicture}
& \text{free}& & & \\ \hline
\begin{tikzpicture}
\draw (0,0) node(e) {$e$};
\draw (1,0) node(Cp) {$C_p$};
\draw (2,0) node(Cp2) {$C_{p^2}$};
\draw (3,0) node (Cp3) {$C_{p^3}$};
\draw[->] (e) -- (Cp);
\draw[->] (Cp) -- (Cp2);
\draw[->] (Cp2) -- (Cp3);
\draw[->] (e) to[bend left] (Cp2);
\draw[->] (e) to[bend left] (Cp3);
\draw[->] (Cp) to[bend right] (Cp3);
\end{tikzpicture}
& \text{free} & & & \\ \hline
\end{array}\]
\end{table}

\begin{table}
\label{Table:Cpq}
\caption{Let $p$ and $q$ be distinct primes and consider the cyclic group $C_{pq}$. The indexing categories for $C_{pq}$ are given in \cite[Figure 2]{Rub20}. The table below describes the Mackey functor underlying $\uA^{\co}[x_{C_{pq}/H}]$.}
\[ \begin{array}{|c||c|c|c|c|}
\hline
 & C_{pq} / e & C_{pq} / C_p & C_{pq} / C_q & C_{pq} / C_{pq} \\ \hline \hline
\begin{tikzpicture}
\draw (0,0) node(e) {$e$};
\draw (-1,0.5) node(Cp) {$C_p$};
\draw (1,0.5) node(Cq) {$C_q$};
\draw (0,1) node (Cpq) {$C_{pq}$};
\end{tikzpicture}
& & & & \text{free} \\ \hline
\begin{tikzpicture}
\draw (0,0) node(e) {$e$};
\draw (-1,0.5) node(Cp) {$C_p$};
\draw (1,0.5) node(Cq) {$C_q$};
\draw (0,1) node (Cpq) {$C_{pq}$};
\draw[->] (e) -- (Cp);
\end{tikzpicture}
& & & & \\ \hline
\begin{tikzpicture}
\draw (0,0) node(e) {$e$};
\draw (-1,0.5) node(Cp) {$C_p$};
\draw (1,0.5) node(Cq) {$C_q$};
\draw (0,1) node (Cpq) {$C_{pq}$};
\draw[->] (e) -- (Cq);
\end{tikzpicture}
& & & & \\ \hline
\begin{tikzpicture}
\draw (0,0) node(e) {$e$};
\draw (-1,0.5) node(Cp) {$C_p$};
\draw (1,0.5) node(Cq) {$C_q$};
\draw (0,1) node (Cpq) {$C_{pq}$};
\draw[->] (e) -- (Cp);
\draw[->] (e) -- (Cq);
\end{tikzpicture}
& & & & \\ \hline
\begin{tikzpicture}
\draw (0,0) node(e) {$e$};
\draw (-1,0.5) node(Cp) {$C_p$};
\draw (1,0.5) node(Cq) {$C_q$};
\draw (0,1) node (Cpq) {$C_{pq}$};
\draw[->] (e) -- (Cp);
\draw[->] (Cq) -- (Cpq);
\end{tikzpicture}
& &  & \text{free} & \\ \hline
\begin{tikzpicture}
\draw (0,0) node(e) {$e$};
\draw (-1,0.5) node(Cp) {$C_p$};
\draw (1,0.5) node(Cq) {$C_q$};
\draw (0,1) node (Cpq) {$C_{pq}$};
\draw[->] (e) -- (Cq);
\draw[->] (Cp) -- (Cpq);
\end{tikzpicture}
& & \text{free} & & \\ \hline
\begin{tikzpicture}
\draw (0,0) node(e) {$e$};
\draw (-1,0.5) node(Cp) {$C_p$};
\draw (1,0.5) node(Cq) {$C_q$};
\draw (0,1) node (Cpq) {$C_{pq}$};
\draw[->] (e) -- (Cpq);
\draw[->] (e) -- (Cp);
\draw[->] (e) -- (Cq);
\end{tikzpicture}
& \text{free} & & & \\ \hline
\begin{tikzpicture}
\draw (0,0) node(e) {$e$};
\draw (-1,0.5) node(Cp) {$C_p$};
\draw (1,0.5) node(Cq) {$C_q$};
\draw (0,1) node (Cpq) {$C_{pq}$};
\draw[->] (e) -- (Cpq);
\draw[->] (e) -- (Cp);
\draw[->] (e) -- (Cq);
\draw[->] (Cq) -- (Cpq);
\end{tikzpicture}
& \text{free} & & & \\ \hline
\begin{tikzpicture}
\draw (0,0) node(e) {$e$};
\draw (-1,0.5) node(Cp) {$C_p$};
\draw (1,0.5) node(Cq) {$C_q$};
\draw (0,1) node (Cpq) {$C_{pq}$};
\draw[->] (e) -- (Cpq);
\draw[->] (e) -- (Cp);
\draw[->] (e) -- (Cq);
\draw[->] (Cp) -- (Cpq);
\end{tikzpicture}
& \text{free} & & & \\ \hline
\begin{tikzpicture}
\draw (0,0) node(e) {$e$};
\draw (-1,0.5) node(Cp) {$C_p$};
\draw (1,0.5) node(Cq) {$C_q$};
\draw (0,1) node (Cpq) {$C_{pq}$};
\draw[->] (e) -- (Cpq);
\draw[->] (e) -- (Cp);
\draw[->] (e) -- (Cq);
\draw[->] (Cq) -- (Cpq);
\draw[->] (Cp) -- (Cpq);
\end{tikzpicture}
& \text{free} & & & \\ \hline
\end{array}\]
\end{table}

\begin{table}
\label{Table:D6}
\caption{Consider the dihedral group $D_6$. This group has five proper subgroups: the trivial subgroup, three conjugate copies of $C_2$, and one copy of $C_3$. Write $H_1, H_2, H_3$ for its subgroups of order two and $C_3$ for its subgroup of order three. The indexing categories for $D_6$ are described in \cite[Figure 4]{Rub20}. 
The table below describes the Mackey functor underlying $\uA^{\co}[x_{D_6/H}]$.}
\[ \begin{array}{|c||c|c|c|c|c|c|}
\hline 
 & D_6 / e & D_6 / H_1 & D_6 / H_2 & D_6/H_3 & D_6/C_3 & D_6 /D_6 \\ \hline \hline
\begin{tikzpicture}
\tikzstyle{every node}=[font=\tiny]
\draw (0,-0.25) node(e) {$e$};
\draw (-1,0.5) node(H1) {$H_1$};
\draw (0,0.5) node(H2) {$H_2$};
\draw (1,0.5) node(H3) {$H_3$};
\draw (0.5,0.6) node(C3) {$C_3$};
\draw (0,1.25) node (D6) {$D_6$};
\end{tikzpicture}
& & & & & &  \text{free} \\ \hline
\begin{tikzpicture}
\tikzstyle{every node}=[font=\tiny]
\draw (0,-0.25) node(e) {$e$};
,0.6) node(C3\draw (-1,0.5) node(H1) {$H_1$};
\draw (0,0.5) node(H2) {$H_2$};
\draw (1,0.5) node(H3) {$H_3$};
\draw (0.5,0.6) node(C3) {$C_3$};
\draw (0,1.25) node (D6) {$D_6$};
\draw[->] (e) -- (C3);
\end{tikzpicture}
& & & & & & \\ \hline
\begin{tikzpicture}
\tikzstyle{every node}=[font=\tiny]
\draw (0,-0.25) node(e) {$e$};
,0.6) node(C3\draw (-1,0.5) node(H1) {$H_1$};
\draw (0,0.5) node(H2) {$H_2$};
\draw (1,0.5) node(H3) {$H_3$};
\draw (0.5,0.6) node(C3) {$C_3$};
\draw (0,1.25) node (D6) {$D_6$};
\draw[->] (e) -- (H1);
\draw[->] (e) -- (H2);
\draw[->] (e) -- (H3);
\end{tikzpicture}
& & & & & & \\ \hline
\begin{tikzpicture}
\tikzstyle{every node}=[font=\tiny]
\draw (0,-0.25) node(e) {$e$};
,0.6) node(C3\draw (-1,0.5) node(H1) {$H_1$};
\draw (0,0.5) node(H2) {$H_2$};
\draw (1,0.5) node(H3) {$H_3$};
\draw (0.5,0.6) node(C3) {$C_3$};
\draw (0,1.25) node (D6) {$D_6$};
\draw[->] (e) -- (H1);
\draw[->] (e) -- (H2);
\draw[->] (e) -- (H3);
\draw[->] (e) -- (C3);
\end{tikzpicture}
& & & & & & \\ \hline
\begin{tikzpicture}
\tikzstyle{every node}=[font=\tiny]
\draw (0,-0.25) node(e) {$e$};
,0.6) node(C3\draw (-1,0.5) node(H1) {$H_1$};
\draw (0,0.5) node(H2) {$H_2$};
\draw (1,0.5) node(H3) {$H_3$};
\draw (0.5,0.6) node(C3) {$C_3$};
\draw (0,1.25) node (D6) {$D_6$};
\draw[->] (e) -- (H1);
\draw[->] (e) -- (H2);
\draw[->] (e) -- (H3);
\draw[->] (C3) -- (D6);
\end{tikzpicture}
& & & & & \text{free} & \\ \hline
\begin{tikzpicture}
\tikzstyle{every node}=[font=\tiny]
\draw (0,-0.25) node(e) {$e$};
,0.6) node(C3\draw (-1,0.5) node(H1) {$H_1$};
\draw (0,0.5) node(H2) {$H_2$};
\draw (1,0.5) node(H3) {$H_3$};
\draw (0.5,0.6) node(C3) {$C_3$};
\draw (0,1.25) node (D6) {$D_6$};
\draw[->] (e) -- (H1);
\draw[->] (e) -- (H2);
\draw[->] (e) -- (H3);
\draw[->] (e) -- (C3);
\draw[->] (e) to[bend left] (D6);
\end{tikzpicture}
& \text{free} & & & & & \\ \hline
\begin{tikzpicture}
\tikzstyle{every node}=[font=\tiny]
\draw (0,-0.25) node(e) {$e$};
,0.6) node(C3\draw (-1,0.5) node(H1) {$H_1$};
\draw (0,0.5) node(H2) {$H_2$};
\draw (1,0.5) node(H3) {$H_3$};
\draw (0.5,0.6) node(C3) {$C_3$};
\draw (0,1.25) node (D6) {$D_6$};
\draw[->] (e) -- (H1);
\draw[->] (e) -- (H2);
\draw[->] (e) -- (H3);
\draw[->] (e) -- (C3);
\draw[->] (e) to[bend left] (D6);
\draw[->] (C3) -- (D6);
\end{tikzpicture}
& \text{free} & & & & & \\ \hline
\begin{tikzpicture}
\tikzstyle{every node}=[font=\tiny]
\draw (0,-0.25) node(e) {$e$};
,0.6) node(C3\draw (-1,0.5) node(H1) {$H_1$};
\draw (0,0.5) node(H2) {$H_2$};
\draw (1,0.5) node(H3) {$H_3$};
\draw (0.5,0.6) node(C3) {$C_3$};
\draw (0,1.25) node (D6) {$D_6$};
\draw[->] (e) -- (H1);
\draw[->] (e) -- (H2);
\draw[->] (e) -- (H3);
\draw[->] (e) -- (C3);
\draw[->] (H1) -- (D6);
\draw[->] (H2) -- (D6);
\draw[->] (H3) -- (D6);
\draw[->] (e) to[bend left] (D6);
\end{tikzpicture}
& \text{free} & & & & & \\ \hline
\begin{tikzpicture}
\tikzstyle{every node}=[font=\tiny]
\draw (0,-0.25) node(e) {$e$};
,0.6) node(C3\draw (-1,0.5) node(H1) {$H_1$};
\draw (0,0.5) node(H2) {$H_2$};
\draw (1,0.5) node(H3) {$H_3$};
\draw (0.5,0.6) node(C3) {$C_3$};
\draw (0,1.25) node (D6) {$D_6$};
\draw[->] (e) -- (H1);
\draw[->] (e) -- (H2);
\draw[->] (e) -- (H3);
\draw[->] (e) -- (C3);
\draw[->] (H1) -- (D6);
\draw[->] (H2) -- (D6);
\draw[->] (H3) -- (D6);
\draw[->] (e) to[bend left] (D6);
\draw[->] (C3) -- (D6);
\end{tikzpicture}
& \text{free} & & & & &  \\ \hline
\end{array}\]
\end{table}

\newpage
\clearpage

\bibliographystyle{plain}
\bibliography{master}

\end{document}